\documentclass[sn-mathphys,Numbered]{sn-jnl}

\usepackage{graphicx,caption}%
\usepackage{subcaption}
\usepackage{multirow}%
\usepackage{amsmath,amssymb,amsfonts}%
\usepackage{amsthm}%
\usepackage{mathrsfs}%
\usepackage[title]{appendix}%
\usepackage{xcolor}%
\usepackage{textcomp}%
\usepackage{manyfoot}%
\usepackage{booktabs}%
\usepackage{algorithm}%
\usepackage{algorithmicx}%
\usepackage{algpseudocode}%
\usepackage{listings}%
\usepackage{enumerate} 

\theoremstyle{thmstyleone}%
\newtheorem{theorem}{Theorem}
\theoremstyle{thmstyletwo}%
\newtheorem{remark}{Remark}%

\theoremstyle{thmstylethree}%
\newtheorem{definition}{Definition}%
\newtheorem{lemma}{Lemma}
\newtheorem{korollar}{Corollary}

\newcommand{\R}{\mathbb{R}}
\newcommand{\N}{\mathbb{N}}

\begin{document}

\title[Article Title]{Geometric Optimization of the First Robin Eigenvalue in Exterior Domains}


\author{\fnm{Lukas} \sur{Bundrock}}

\affil{\orgdiv{Department of Mathematics}, \orgname{The University of Alabama}, \orgaddress{\city{Tuscaloosa}, \postcode{AL 35487}, \country{USA},} {lbundrock@ua.edu}}


\abstract{This paper addresses the geometric optimization problem of the first Robin eigenvalue in exterior domains, specifically the lowest point of the spectrum of the Laplace operator under Robin boundary conditions in the complement of a bounded domain. In contrast to the Laplace operator on bounded domains, the spectrum of this operator is not purely discrete. The discrete nature of the first eigenvalue depends on the parameter of the Robin boundary condition. In two dimensions, D. Krejcirik and V. Lotoreichik show that the ball maximizes the first Robin eigenvalue among all smooth, bounded, simply connected sets with given perimeter or given area, provided the eigenvalue is discrete.

We extend these findings to higher dimensions. The discrete spectrum of the Laplace operator under Robin boundary conditions can be characterized through the Steklov eigenvalue problem in exterior domains, a topic studied by G. Auchmuty and Q. Han. Assuming that the lowest point of the spectrum is a discrete eigenvalue, we show that the ball is a local maximizer among nearly spherical domains with prescribed measure. However, in general, the ball does not emerge as the global maximizer for the first Robin eigenvalue under either prescribed measure or prescribed perimeter.

}


\maketitle

\section{Introduction}\label{sec1}
Finding the optimal shape of an object to make it as efficient as possible is a naturally occurring question. In mathematical terms, this means to identify the set that minimizes a given functional among all sets that satisfy given constraints. Specifically, our focus lies in the optimization of eigenvalues of elliptic operators. In this paper, we are concerned with optimizing the lowest point of the spectrum of the Robin Laplacian,
\begin{align}\label{eq:varcharint}
\lambda_1^\alpha(\Omega) := \inf_{0 \neq u \in W^{1,2}(\Omega)}  \frac{\int_{\Omega} | \nabla u |^2 \,  \mathrm{d}x + \alpha \int_{\partial \Omega}|u|^2 \,  \mathrm{d}S}{\int_{\Omega} |u|^2 \,  \mathrm{d}x}.
\end{align}
For bounded domains, the situation is well known. For $n\geq 2$, $\alpha >0$, D. Daners shows in \cite{Daners2006} that the ball $B_R$ minimizes $\lambda_1^\alpha$ among all bounded Lipschitz domains $\Omega \subseteq \mathbb{R}^n $ with $|\Omega| = | B_R|$. For $\alpha < 0$, M. Bareket conjectures in \cite{Bareket1977} that the ball maximizes the first eigenvalue among all smooth, bounded domains with prescribed measure. This is proven not to be true in general. A counterexample is given in \cite[Theorem 1]{FREITAS2015322} by P. Freitas and D. Krejcirik. However, this counterexample does not prove the statement false if $\alpha$ is close to zero. Furthermore, it is shown in \cite[Theorem 2]{FREITAS2015322} that, for sufficiently small $| \alpha |$, the ball is a maximizer among all bounded, planar $\mathcal{C}^2$-domains.

In electromagnetic and gravitational field theories, questions often arise about the solution of boundary value problems involving the Laplacian, which includes the occurrence of such problems on unbounded domains. Thus, we are interested in studying whether similar inequalities apply to the Robin eigenvalue on the complement of a bounded domain. We define $\Omega^\text{ext} := \mathbb{R}^n \setminus \overline{\Omega}$.  The outer unit normal, pointing out of $\Omega$, is denoted by $\nu$, and $\partial_\nu u = \langle \nabla u, \nu \rangle$ is the derivative in normal direction. Throughout the paper, we assume that $\Omega^\text{ext} $ is connected. We consider the eigenvalue problem
\begin{align}\label{eq:RobinExt}
\begin{cases}
\Delta {u} + \lambda {u} = 0  & \, \text{ in } \Omega^\text{ext}, \\
-\partial_{\nu} {u}  + \alpha u=0 & \,  \text{ on } \partial \Omega,
\end{cases}
\end{align}
where $\alpha \in \mathbb{R}$. We understand \eqref{eq:RobinExt} as a spectral problem
for the self-adjoint operator  associated with the bilinear form
\begin{align}\label{eq:bilineExt}
a: W^{1,2}(\Omega^\text{ext}) \times W^{1,2}(\Omega^\text{ext}) \to \mathbb{R}, \, \, a(u,v) := \int_{\Omega^\text{ext}} \langle \nabla u, \nabla v \rangle \, \mathrm{d}x + \alpha \int_{\partial \Omega} u v \,  \mathrm{d}S.
\end{align}
This operator is called the Robin Laplacian in exterior domains and is denoted by $- \Delta_\alpha^{\Omega^\text{ext}}$. In contrast to bounded domains, the embedding $W^{1,2}(\Omega^\text{ext}) \hookrightarrow L^2(\Omega^\text{ext})$ is not compact. Thus, we can neither conclude that $-\Delta_\alpha^{\Omega^\text{ext}}$ has compact resolvent nor a purely discrete spectrum. Indeed, in \cite[Proposition 1]{krejcirik2016optimisation}, D. Krejcirik and V. Lotoreichik show that there is a nonempty essential spectrum of $-\Delta_\alpha^{\Omega^\text{ext}}$, given by
\begin{align*}
\sigma_{\text{ess}} \left( - \Delta_\alpha^{\Omega^\text{ext}} \right) = [0, \infty)
\end{align*} 
for all $\alpha \in \mathbb{R}$, $n \geq 2$ and for all bounded, smooth domains $\Omega$ such that $\Omega^\text{ext}$ is connected. However, it still makes sense to characterize the lowest point of the spectrum by
\begin{align}\label{eq:varcharext}
\lambda_1^\alpha(\Omega^\text{ext}) := \inf_{0 \neq u \in W^{1,2}(\Omega^\text{ext})}  \frac{\int_{\Omega^\text{ext}} | \nabla u |^2 \,  \mathrm{d}x + \alpha \int_{\partial \Omega}|u|^2 \,  \mathrm{d}S}{\int_{\Omega^\text{ext}} |u|^2 \,  \mathrm{d}x}.
\end{align}
If the discrete spectrum is not empty, $\lambda_1^\alpha(\Omega^\text{ext})$ is negative and coincides with the lowest eigenvalue. If there is no discrete eigenvalue, then $\lambda_1^\alpha(\Omega^\text{ext})=0$, where in this case we do not know if the infimum is attained. Provided $\lambda_1^\alpha(\Omega^\text{ext})$ is a discrete eigenvalue, D. Krejcirik and V. Lotoreichik show in \cite[Corollary 5]{krejvcivrik2020optimisation} that the ball maximizes $\lambda_1^\alpha(\Omega^\text{ext})$ in dimension $n=2$ among all smooth, bounded, simply connected open sets with given measure and among all smooth, bounded, simply connected open sets with given perimeter. In dimension $n \geq 3$ the ball does not maximizes $\lambda_1^\alpha(\Omega^\text{ext})$ anymore. Thus, for dimension $n \geq 3$, they define $\mathcal{M}(\Omega) := \frac{1}{| \partial \Omega|} \int_{\partial \Omega} M^{n-1} \, \mathrm{d}S$, where $M$ denotes the mean curvature of $\partial \Omega$, and show in  \cite[Theorem 6]{krejvcivrik2020optimisation} that the ball maximizes the first Robin eigenvalue among all convex, smooth, bounded open sets with $ \mathcal{M}(\Omega) = \mathcal{M}(B_R)$.

We show that the ball, although no longer the global maximizer, is a local maximizer of $\lambda_1^\alpha(\Omega^\text{ext})$ among all domains with prescribed measure in any dimension.

First we characterize the existence of a discrete eigenvalue of the Robin Laplacian in exterior domains. To this end, we consider for $n \geq 3$, the Steklov eigenvalue problem
\begin{align}\label{eq:StekExt}
\begin{cases}
\Delta {u}  = 0  & \, \text{ in } \Omega^\text{ext}, \\
-\partial_{\nu} {u}  = \mu {u} & \,  \text{ on } \partial \Omega.
\end{cases}
\end{align}
In \cite{Auchmuty2014}, G. Auchmuty and Q. Han illustrate that solutions of \eqref{eq:StekExt} are in general not in $L^2(\Omega^\text{ext})$. Thus, they construct the Hilbert space $E^1(\Omega^\text{ext})$, the space of finite energy functions. Considering \eqref{eq:StekExt} in $E^1(\Omega^\text{ext})$, we obtain a well posed problem. 


In Section \ref{subsc:charal}, we characterize the discrete spectrum of the Robin Laplacian in exterior domains, depending on $\alpha$. We show that $\lambda_1^\alpha(\Omega^\text{ext})$ is a discrete eigenvalue if and only if $\alpha < - \mu_1(\Omega^\text{ext})$, where $\mu_1(\Omega^\text{ext})$ denotes the first eigenvalue of \eqref{eq:StekExt}. Provided that $\lambda_1^\alpha(\Omega^\text{ext})$ is a discrete eigenvalue, we are interested in optimizing $\lambda_1^\alpha(\Omega^\text{ext})$  among all domains of given measure or given perimeter. 

Before we discuss this geometric optimization problem, we demonstrate in Section \ref{sec:cont} the continuity of $\lambda_1^\alpha(\Omega^\text{ext})$ and $\mu_1(\Omega^\text{ext})$ with respect to the Hausdorff metric, which implies the existence of an optimal domain in a suitable collection of subsets of $\mathbb{R}^n$. Furthermore, we establish in Section \ref{sec:mon}, a monotonicity result for $\lambda_1^\alpha(\Omega^\text{ext})$ and $\mu_1(\Omega^\text{ext})$ concerning a specific type of domain inclusion. This implies that achieving optimality within a given perimeter results in optimality within a given measure.

Comparing $\lambda_1^\alpha(B^\text{ext})$ with $\lambda_1^\alpha(E^\text{ext})$, where $B$ is a ball and $E$ is an ellipsoid with equal measure, we demonstrate in Section \ref{sec:glob} that the ball cannot be the global maximizer of $\lambda_1^\alpha(\Omega^\text{ext})$ among all smooth domains with given measure. Leveraging the monotonicity discussed in Section \ref{sec:mon}, this extends to the ball not being the global maximizer among smooth domains with a prescribed perimeter.

Lacking a global approach, we study how local perturbations affect the eigenvalue. By proceeding similar to \cite{Wagner1}, as described in Section \ref{sec:DomVarInt}, we derive formulas for the first and second variations of $\lambda_1^\alpha(\Omega^\text{ext})$ in  Section \ref{subsubsec:Eig}. Using the second variation, we show in Section \ref{signofsec} that the ball maximizes $\lambda_1^\alpha(\Omega^\text{ext})$ locally among all nearly spherical domains with prescribed measure. To accomplish this, we establish in section \ref{sec:ineqbess} a new inequality, to the best of our knowledge, for the ratio of Bessel functions. In Section \ref{sec:quantext}, we derive a quantitative inequality.

Some results of this paper have been published previously in \cite{BundPhD}.

\section{Basic Properties of the Robin Eigenvalue in Exterior Domains}\label{sec:1}
If $\Omega$ is a bounded domain, it holds $\lambda_1^\alpha(\Omega) < 0$ if and only if $\alpha < 0$. This can easily be seen by inserting a constant test function into \eqref{eq:varcharint}. Thus, one might guess that for $\alpha < 0$ it holds $\lambda_1^\alpha(\Omega^{\text{ext}}) < 0$. It turns out that this condition is not sufficient in general. In \cite[Proposition 2]{krejvcivrik2020optimisation}, it is shown that there exists a nonpositive constant $\alpha^*(\Omega^{\text{ext}})$ such that $\lambda_1^\alpha(\Omega^{\text{ext}})$ is a discrete, negative eigenvalue of \eqref{eq:RobinExt} if and only if $ \alpha < \alpha^*(\Omega^{\text{ext}})$. More precisely, it holds
\begin{align*}
\begin{cases}
\alpha^*(\Omega^{\text{ext}}) = 0 \text{ for } n \in \{ 1,2 \},\\
\alpha^*(\Omega^{\text{ext}}) < 0 \text{ for } n \geq 3.
\end{cases}
\end{align*}
If $\alpha < \alpha^*(\Omega^{\text{ext}})$, i.e. $\lambda_1^\alpha(\Omega^{\text{ext}})$ is a discrete eigenvalue, it can be shown by standard methods that $\lambda_1^\alpha(\Omega^{\text{ext}})$ is simple and the corresponding eigenfunction can be chosen strictly positive. Since these properties are important for further calculations, we start by characterizing $\alpha^*(\Omega^{\text{ext}})$ for $n \geq 3$ in more detail. The results of Section \ref{subsc:charal} have been published beforehand in \cite[Section 2]{BundPhD}.

\subsection{Characterization of $\alpha^*(\Omega^{\text{ext}})$}\label{subsc:charal}

As in \citep[Proposition 2.1]{AntunesFK}, it can be shown that $\lambda_1^\alpha(\Omega^{\text{ext}})$ is continuous in $\alpha$. Combined with $\lambda_1^\alpha(\Omega^{\text{ext}}) < 0 \Leftrightarrow \alpha < \alpha^*(\Omega^{\text{ext}})  $, this means that
\begin{align*}
\lambda_1^{\alpha^*(\Omega^{\text{ext}})}(\Omega^{\text{ext}}) = \inf_{0 \neq u \in W^{1,2}(\Omega^{\text{ext}})} \frac{\int_{\Omega^{\text{ext}}} | \nabla u |^2 \,  \mathrm{d}x + \alpha^*(\Omega^{\text{ext}}) \int_{\partial \Omega}|u|^2\,  \mathrm{d}S}{\int_{\Omega^{\text{ext}}} |u|^2 \,  \mathrm{d}x} = 0.
\end{align*}
Because $0 \in \sigma_\text{ess}\left( - \Delta_\alpha^{\Omega^\text{ext}} \right)$, we don't know if the infimum is attained. Nevertheless, for all $ u \in W^{1,2}(\Omega^{\text{ext}})$, it holds
\begin{align*}
 0 \leq \int_{\Omega^{\text{ext}}} | \nabla u |^2 \,  \mathrm{d}x + \alpha^*(\Omega^{\text{ext}}) \int_{\partial \Omega}|u|^2\,  \mathrm{d}S,
\end{align*}
which suggests that $-\alpha^*(\Omega^{\text{ext}})$ could potentially be a Steklov eigenvalue. 

A. Auchmuty and Q. Han consider in \cite{Auchmuty2014} the exterior harmonic Steklov eigenvalue problem \eqref{eq:StekExt} for dimensions $n \geq 3$. For $\Omega = B_R$, the radial solutions of \eqref{eq:StekExt} are of the form $u(x) = c_1 + c_2 |x|^{2-n}$. Here, $u \in L^2(B_R^{\text{ext}})$ does not hold true in general. To solve this problem, they define the space of  {finite energy} functions $E^1(\Omega^{\text{ext}})$ to be the subspace of $L^1_{\text{loc}}(\Omega^\text{ext})$ of functions that satisfy
\vspace{1ex}
\begin{enumerate}
\item $\nabla u \in L^2(\Omega^{\text{ext}};\R^n)$,
\item $u$ decays at infinity which means $S_c := \{ x \in \Omega^{\text{ext}}: | u(x) | \geq c \}$ has finite measure for all $c > 0$.
\end{enumerate}
\vspace{1ex}
In addition, $E^1(\Omega^\text{ext})$ is a Hilbert space with respect to the inner product
\begin{align*}
\langle f,g \rangle_{E^1(\Omega^{\text{ext}})} := \int_{\Omega^{\text{ext}}} \langle \nabla f, \nabla g \rangle \,  \mathrm{d}x + \frac{1}{| \partial \Omega|} \int_{\partial \Omega} f g\,  \mathrm{d}S.
\end{align*}
Furthermore, they prove that the boundary trace operator $\gamma: E^1(\Omega^{\text{ext}}) \to L^2(\partial \Omega)$ is compact. Using the methods of \cite[Chapter 8]{GilbargTrudinger}, it can be shown that the first Steklov eigenvalue $\mu_1$ is simple and its corresponding eigenfunction $u_1 \in E^1(\Omega^{\text{ext}})$ is of constant sign. Moreover, since $\gamma: E^1(\Omega^{\text{ext}}) \to L^2(\partial \Omega)$ is compact, the spectrum is purely discrete and we have a sequence of eigenvalues
\begin{align*}
\mu_1(\Omega^{\text{ext}}) \leq \mu_2(\Omega^{\text{ext}}) \leq \ldots,
\end{align*}
accumulating at $\infty$. On the basis of these properties, we can characterize $\alpha^*(\Omega^{\text{ext}})$.
\vspace{1ex}
\begin{theorem}\label{satz:alpha*var'}
Let $\mu_1(\Omega^{\text{ext}})$ be the first eigenvalue of \eqref{eq:StekExt}. Then, $ \alpha^*(\Omega^{\text{ext}}) = -\mu_1(\Omega^{\text{ext}})$.
\end{theorem}
\begin{proof}
To prove the theorem, we show the two inequalities $\alpha^*(\Omega^{\text{ext}}) \leq -\mu_1(\Omega^{\text{ext}})$ and $\alpha^*(\Omega^{\text{ext}}) \geq -\mu_1(\Omega^{\text{ext}})$.

Beginning with the first inequality, we can characterize $\mu_1(\Omega^{\text{ext}})$ as
\begin{align}\label{eq:varcharstek}
\mu_1(\Omega^{\text{ext}}) = \inf_{\substack{u \in E^1(\Omega^{\text{ext}}), \\ ||u||^2_{L^2(\partial \Omega)}=1}}   || \nabla u ||^2_{L^2(\Omega^{\text{ext}})}
\end{align}
and since $W^{1,2}(\Omega^{\text{ext}}) \subseteq  E^1(\Omega^{\text{ext}})$, we immediately obtain
\begin{align*}
-\mu_1(\Omega^{\text{ext}}) \geq -\inf_{\substack{u \in W^{1,2}(\Omega^{\text{ext}}), \\ ||u||^2_{L^2(\partial \Omega)}=1}}   || \nabla u ||^2_{L^2(\Omega^{\text{ext}})}.
\end{align*}
This implies $\alpha^*(\Omega^{\text{ext}}) \leq -\mu_1(\Omega^{\text{ext}})$.

To prove the reversed inequality, we approximate the first eigenfunction of the Steklov eigenvalue using functions with compact support. Let $u_1$ be the first eigenfunction of \eqref{eq:StekExt}, with $u_1 \in E^1(\Omega^{\text{ext}})$ and $||u_1||_{L^2(\partial \Omega)} = 1$. Then,
\begin{align*}
0 = \int_{\Omega^{\text{ext}}} | \nabla u_1  |^2 \,  \mathrm{d}x -\mu_1(\Omega^\text{ext}).
\end{align*}
To approximate $u_1$ by smooth functions with compact support, we utilize a result from G. Lu and B. Ou, proven in \cite[Proof of Theorem 5.2, Proposition 2.2]{LuOu}. For a function $u \in W^{1,2}_{\text{loc}}(\Omega^{\text{ext}})$ with $| \nabla u | \in L^2(\Omega^{\text{ext}})$, they define
\begin{align*}
 \left( u \right)_\infty := \lim_{R \to \infty} \frac{1}{|\Omega^{\text{ext}} \cap B_R|} \int_{\Omega^{\text{ext}} \cap B_R} u \,  \mathrm{d}x
\end{align*}
and show that  $w := u - \left( u \right)_\infty$ can be approximated by smooth functions, i.e.  for every $\varepsilon > 0$ there exists $\psi_R \in C_0^\infty(\mathbb{R}^n)$  with $\psi_R(x) = 1$ for $|x|<R$, such that
\begin{align*}
\int_{\Omega^{\text{ext}}} | \nabla w - \nabla (w \,  \psi_R )|^2  \,  \mathrm{d}x < \varepsilon.
\end{align*}
Since $E^1(\Omega^{\text{ext}}) \subset  L^\frac{2n}{n-2}(\Omega^{\text{ext}})$ (see \cite[Corollary 3.4]{Auchmuty2014}), we obtain
\begin{align*}
\frac{ \int_{\Omega^{\text{ext}} \cap B_R} |u_1| \,  \mathrm{d}x}{|\Omega^{\text{ext}} \cap B_R|} &\leq \frac{1}{|\Omega^{\text{ext}} \cap B_R|} \left(\int_{\Omega^{\text{ext}} \cap B_R} |u_1|^\frac{2n}{n-2} \,  \mathrm{d}x \right)^\frac{n-2}{2n} \left(\int_{\Omega^{\text{ext}} \cap B_R} 1 \,  \mathrm{d}x \right)^\frac{n+2}{2n} \\
&= |\Omega^{\text{ext}} \cap B_R|^{\frac{n+2}{2n}-1} \left(\int_{\Omega^{\text{ext}} \cap B_R} |u_1|^\frac{2n}{n-2} \,  \mathrm{d}x \right)^\frac{n-2}{2n}.
\end{align*}
Since $\frac{n+2}{2n}-1 < 0$, it holds $\lim_{R \to \infty} |\Omega^{\text{ext}} \cap B_R|^{\frac{n+2}{2n}-1} = 0$. Also, since $u_1  \in L^\frac{2n}{n-2}(\Omega^{\text{ext}})$, we have
\begin{align*}
 \lim_{R \to \infty}\left(\int_{\Omega^{\text{ext}} \cap B_R} |u_1|^\frac{2n}{n-2} \,  \mathrm{d}x \right)^\frac{n-2}{2n} = ||u_1||_{L^\frac{2n}{n-2}(\Omega^{\text{ext}})}.
\end{align*}
Thus, $u_1 - \left( u_1 \right)_\infty = u_1$ can be approximated by $u_1 \psi_R$. By choosing $R$ large enough, such that $\phi(x) = 1 $ for all $x \in \partial \Omega$, we have
\begin{align*}
\int_{\partial \Omega} (u_1 \, \psi_R )^2\,  \mathrm{d}S = \int_{\partial \Omega} u_1^2\,  \mathrm{d}S = 1.
\end{align*}
For  $\alpha < -\mu_1(\Omega^\text{ext})$, we can choose $R$ large enough such that $\phi := u_1 \,  \psi_R \in W^{1,2}(\Omega^{\text{ext}})$ satisfies
\begin{align*}
\int_{\Omega^{\text{ext}}} | \nabla \phi  |^2 \,  \mathrm{d}x -\mu_1(\Omega^\text{ext}) = \int_{\Omega^{\text{ext}}} | \nabla \phi  |^2  - | \nabla u_1  |^2 \,  \mathrm{d}x <  \frac{-\mu_1(\Omega^\text{ext}) - \alpha}{2} =: \widehat{\varepsilon}(\alpha).
\end{align*}
Therefore, we obtain
\begin{align*}
\int_{\Omega^{\text{ext}}} | \nabla \phi  |^2 \,  \mathrm{d}x + \alpha \int_{ \partial \Omega} | \phi  |^2\,  \mathrm{d}S < \widehat{\varepsilon}(\alpha) +\mu_1(\Omega^\text{ext}) + \alpha = \frac{\mu_1(\Omega^\text{ext}) + \alpha}{2} < 0.
\end{align*}
Thus, $\alpha < -\mu_1(\Omega^\text{ext})$ implies $\lambda_1^\alpha(\Omega^{\text{ext}}) < 0$. Hence, $\alpha^*(\Omega^{\text{ext}}) = - \mu_1(\Omega^{\text{ext}})$.
\end{proof}

Analogous calculations for higher eigenvalues lead to the following corollary.

\vspace{1ex}
\begin{korollar}\label{satz:szek-rob-zusammenhang}
Let $\mu_k(\Omega^{\text{ext}})$ denote the $k$-th eigenvalue of \eqref{eq:StekExt}. The Robin Laplacian has $k$ discrete eigenvalues if and only if $\alpha < - \mu_k(\Omega^{\text{ext}})$.
\end{korollar}
\vspace{1ex}

As an example, we consider $\Omega=B_R$. Here, it is possible to find an explicit formula for $\alpha^*(\Omega^{\text{ext}})$ (see \cite[Proposition 3]{krejvcivrik2020optimisation}): For $\alpha < \alpha^*(B_R^{\text{ext}} )$ the first eigenvalue $\lambda_1^\alpha(B_R^{\text{ext}})$ is negative and simple, so the corresponding eigenfunction has to be radial. Thus, \eqref{eq:RobinExt} becomes
\begin{align}\label{eq:problemB_Rext}
\begin{cases}
u''(r) + \frac{n-1}{r} u'(r) + \lambda {u}(r) = 0 \,  & \text{ for } r \in (R,\infty), \\
-{u'}(r) + \alpha {u}(r) = 0 \,  &\text{ for } r=R,
\end{cases}
\end{align}
where $u'$ denotes the derivative of $u$ with respect to $r$. For $n \geq 2$, the solution of \eqref{eq:problemB_Rext} is given by
\begin{align}\label{eq:uExtBall}
u(r)= c \, r^{-\frac{n-2}{2}} K_{\frac{n-2}{2}}(r \sqrt{-\lambda}), \quad c \in \mathbb{R},
\end{align}
where $K_m$ denotes the modified Bessel function of second kind (see Appendix \ref{secA1}). The boundary condition yields a relation between $\alpha$ and $\lambda$. Using \eqref{eq:Bessediff}, we obtain
\begin{align*}
\alpha &= \frac{u'(R)}{u(R)} 
= -\sqrt{-\lambda} \frac{  K_{\frac{n}{2}}(R \sqrt{-\lambda}) }{  K_{\frac{n-2}{2}}(R \sqrt{-\lambda})}.
\end{align*}
This equation indeed provides a unique relation between $\alpha$ and $\lambda$. To prove this, we define for $m = \frac{k}{2}, \, k \in \mathbb{N}$, the mapping
\begin{align*}
f_m:(0,\infty) \to \R, \, \, z \mapsto -z \frac{  K_{m+1}(z) }{  K_{m}(z)}
\end{align*}
and show that $f_m$ is strictly monotonically decreasing. Since $R \, \alpha = f_{\frac{n-2}{2}}\left( R \sqrt{-\lambda} \right)$, this yields the desired uniqueness between $\alpha$ and $\lambda$. Using \eqref{eq:Bessediff} and \eqref{eq:RecK}, we obtain
\begin{align*}
\frac{\mathrm{d}}{\mathrm{d}z}f_m(z) = \frac{z^2 + 2m \frac{zK_{m+1}(z)}{K_m(z)} - \left(\frac{zK_{m+1}(z)}{K_m(z)} \right)^2}{z}.
\end{align*}
From \eqref{eq:segura}, we obtain $\frac{zK_{m+1}(z)}{K_m(z)} \geq m + \sqrt{m^2+z^2} > 2m$. Furthermore, the mapping $x \mapsto 2mx-x^2$ is strictly monotonically decreasing on $(m, \infty)$ and we conclude
\begin{align}\label{eq:mondecf}
\frac{\mathrm{d}}{\mathrm{d}z}f_m(z) &< \frac{z^2 + 2m \left( m + \sqrt{m^2+z^2} \right) - \left( m + \sqrt{m^2+z^2} \right)^2}{z} = 0.
\end{align}
Thus, the mapping
\begin{align*}
\lambda \mapsto-\sqrt{-\lambda} \frac{  K_{\frac{n}{2}}(R \sqrt{-\lambda}) }{  K_{\frac{n-2}{2}}(R \sqrt{-\lambda})}
\end{align*}
is strictly monotonically increasing for $\lambda \in (-\infty,0)$. For $R=1$ and different $n$, this relation is shown below in Figure \ref{fig:f1}. Since it seems more natural to express $\lambda$ depending on $\alpha$, we also plot the inverse relation in Figure \ref{fig:f-1}.
\begin{figure}[h]
\begin{subfigure}{0.5\textwidth}
\centering
\includegraphics[trim = 10mm 110mm 50mm 20mm, width=\textwidth]{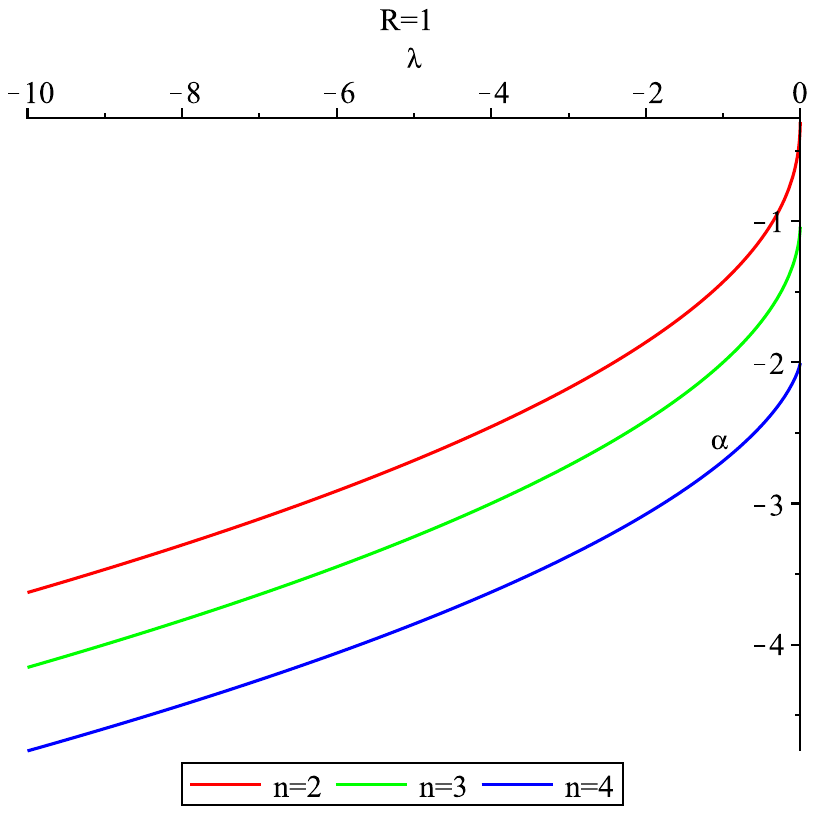}
\caption{$\alpha$ depending on $\lambda_1$}
\label{fig:f1}
\end{subfigure}
\begin{subfigure}{0.5\textwidth}
\centering
\includegraphics[trim = 10mm 110mm 50mm 20mm, width=\textwidth]{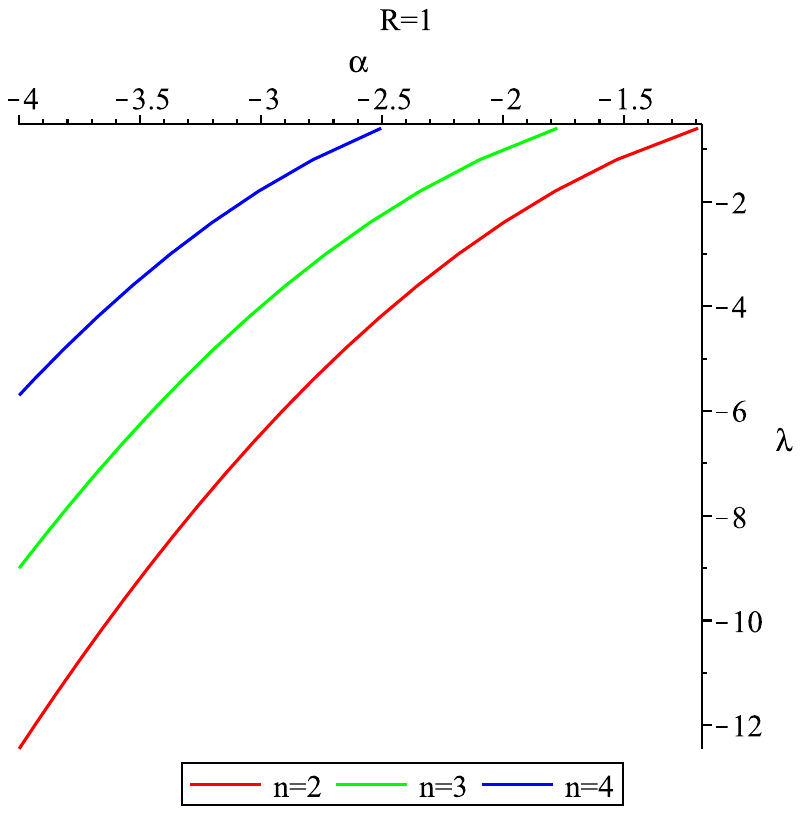}
\caption{$\lambda_1$ depending on $\alpha$}
\label{fig:f-1}
\end{subfigure}
\caption{Relation between $\alpha$ and $\lambda_1$ for the $n$-dimensional unit ball.}
\label{fig:image2}
\end{figure}

To determine $\alpha^*$, we have to determine $\lim_{\lambda \to 0} \alpha(\lambda)$. We use the asymptotic behavior of the Bessel functions, \eqref{eq:BesselKasymp0} and \eqref{eq:BesselKasympinf} and obtain for $n \geq 3$, 
\begin{align*}
\lim_{\lambda \to 0} \, \alpha(\lambda) &= \lim_{\lambda \to 0} \, \frac{f_{\frac{n-2}{2}}(R \sqrt{-\lambda})}{R} = \lim_{\lambda \to 0} \, -\sqrt{-\lambda} \frac{  \frac{\Gamma(\frac{n}{2}) 2^{\frac{n-2}{2}}}{\left( R \sqrt{-\lambda}) \right)^\frac{n}{2}} }{  \frac{\Gamma(\frac{n-2}{2}) 2^{\frac{n-4}{2}}}{\left( R \sqrt{-\lambda}) \right)^\frac{n-2}{2}} } = -\frac{n-2}{R}
\end{align*}
and in view of the monotonicity, we obtain $\lambda <0  \,\Leftrightarrow  \, \alpha < -\frac{n-2}{R}$, which implies
\begin{align}\label{eq:asymptoticsalpha}
\alpha^*(B_R^{\text{ext}}) = - \frac{n-2}{R}.
\end{align}
On the other hand, for $\Omega = B_R$, the first eigenvalue of \eqref{eq:StekExt} is given by $\mu_1(B_R^\text{ext}) = \frac{n-2}{R}$ which confirms Theorem \ref{satz:alpha*var'}.

\subsection{Continuity}\label{sec:cont}
In order to show existence of optimal or critical domains, it is often helpful to have continuity of the eigenvalue. In this section, we show that the mappings $\Omega \mapsto \lambda_1^\alpha(\Omega^\text{ext})$ and $\Omega \mapsto \mu_1(\Omega^\text{ext})$ are continuous with respect to the Hausdorff metric. In particular, in view of Theorem \ref{satz:alpha*var'}, the continuity of $\mu_1(\Omega^\text{ext})$ allows us to deduce that $\alpha < \alpha^*(\Omega_0^\text{ext})$ implies $\alpha < \alpha^*(\Omega^\text{ext})$ for all domains $\Omega$ in a neighborhood of $\Omega_0$.

\vspace{1ex}
\begin{definition}\label{def:HausX}
Let $D \subseteq \mathbb{R}^n$ be a given bounded subset of $\mathbb{R}^n$. Then, we define
\begin{align*}
X:= \left\{ \Omega \subseteq D, \, \Omega \, \text{is a convex domain with } \,  \partial \Omega \in \mathcal{C}^2 \right\}.
\end{align*}
We denote the Hausdorff metric of two domains $A, B \in X$ by
\begin{align*}
\delta(A,B) := \inf \left\{ \varepsilon > 0: A \subseteq B_\varepsilon \text{ and } B \subseteq A_\varepsilon \right\},
\end{align*}
where $B_\varepsilon := \cup_{b \in B} \{ z \in \R^n: ||z-b|| \leq \varepsilon \}$. 
\end{definition}
\vspace{1ex}

The definition of the Hausdorff metric requires less regularity of the boundary, but for now we are only interested in the smooth setting. To prove continuity with respect to the Hausdorff metric, we need the following extension result.

\vspace{1ex}
\begin{lemma}\label{lemma:extension}
Let $\Omega_0 \in X$ and let $\left( \Omega_m \right)_{m \in \N} \subset X$ such that $\lim_{m \to \infty} \delta( \Omega_m, \Omega_0) = 0$. If $ \left( u_m \right)_{m \in \N} \subset W^{1,2}(\Omega_m^{\text{ext}})$  fulfills $|| u_m ||_{W^{1,2}(\Omega_m^{\text{ext}})} \leq K_1$ (independent of $m$), there exist extensions $\widehat{u_m} \in W^{1,2}(\R^n)$ with $u_m(x) = \widehat{u_m}(x)$ for $x \in \Omega_m^{\text{ext}}$ and $|| \widehat{u_m}||_{W^{1,2}(\R^n)} \leq K_2$ (independent of $m$).
\end{lemma}
\begin{proof}
We choose $R,r>0$ such that $\overline{B_r} \subset \Omega_m$ and $\overline{\Omega} \subset B_R$  for all $m \in \N_0$. We consider the bounded domains $\widehat{\Omega_m} := \Omega_m^{\text{ext}} \cap B_R$. Since $\Omega_m$ is convex, $\widehat{\Omega_m}$ fulfills the cone property, given in \cite[Definition 2]{chenais1975existence}.  Hence, \cite[Theorem II.1]{chenais1975existence} yields that there is a uniform extension $\widetilde{u_m}$ of $ u_m \big|_{\widehat{\Omega_m}}$ to $\R^n$,  i.e. there exists a sequence $\widetilde{u_m} \in W^{1,2}(\R^n)$ with $|| \widetilde{u}_m ||_{W^{1,2}(\R^n)} \leq K_3$ and $\widetilde{u_m}(x) = u_m(x)$ for $x \in \widehat{\Omega_m}$. Thus, we define
\begin{align*}
\widehat{u_m}(x) := \begin{cases}
\widetilde{u_m}(x) &\text{ for } x \in B_R,\\
u_m(x) &\text{ for } x \in \Omega_m^{\text{ext}},
\end{cases}
\end{align*}
and obtain $
|| \widehat{u_m} ||^2_{W^{1,2}(\R^n)} \leq  || \widehat{u_m} ||^2_{W^{1,2}(B_R)} + || \widehat{u_m} ||^2_{W^{1,2}(\Omega_m^{\text{ext}})} \leq K_3^2 + K_1^2 =: K_2^2$.
\end{proof}

Using the previous lemma, we can proceed similar to D. Bucur and F. Gazzola in \cite[Theorem 5.1]{bucur2011first} to prove the continuity of the Steklov eigenvalue.
\vspace{1ex}
\begin{lemma}\label{lemma:contf1}
Let $n \geq 3$. The mapping
\begin{align*}
f_1: X \to \R, \, \, \Omega \mapsto \mu_1(\Omega^{\text{ext}})
\end{align*}
is continuous with respect to the Hausdorff metric.
\end{lemma}
\begin{proof}
The proof is structured as the proof of \cite[Theorem 5.1]{bucur2011first}: We consider a sequence $ \left( \Omega_m \right)_{m \in \N} \subset X $ that converges in the Hausdorff metric. We start by proving the upper semicontinuity. This gives an upper bound for the sequence of eigenfunctions corresponding to $\mu_1(\Omega_m^\text{ext})$, which allows us to deduce the existence of a weakly convergent subsequence. This can be used to infer the convergence of $\mu_1(\Omega_m^\text{ext})$.

We start by proving the upper semicontinuity. Let $ \left( \Omega_m \right)_{m \in \N} \subset X $ be a sequence with $\lim_{m \to \infty} \delta(\Omega_m, \Omega_0)=0$ and we choose $r,R>0$ such that $\overline{B_r} \subseteq \Omega_m, \, \overline{\Omega_m} \subseteq B_R$ for all $m \in \mathbb{N}_0$.  In addition, we consider a sequence $(t_m)_{m \in \N} \subset \mathbb{R}$ with $\lim_{m \to \infty} t_m = 1$ and $ \overline{\Omega_0} \subseteq t_m \Omega_m $. Then, $ (t_m \Omega_m)^\text{ext} \subseteq \Omega_0^\text{ext}$ and any $u \in E^1(\Omega_0^\text{ext})$ fulfills
\begin{align*}
\lim_{m \to \infty} \int_{(t_m \Omega_m)^\text{ext}} | \nabla u|^2 \, \mathrm{d} x = \int_{\Omega_0^\text{ext}} | \nabla u|^2 \, \mathrm{d} x \, \, \text{ and } \, \, \lim_{m \to \infty} \int_{ \partial (t_m \Omega_m)} u^2 \, \mathrm{d} x = \int_{ \partial \Omega_0} u^2 \, \mathrm{d} x.
\end{align*} 
Hence, for any $u \in E^1(\Omega_0^\text{ext})$, it holds
\begin{align*}
\limsup_{m \to \infty} \mu_1 ( (t_m \Omega_m)^\text{ext} ) \leq \lim_{m \to \infty} \frac{||\nabla u||_{L^2((t_m \Omega_m)^\text{ext})}^2}{||u||_{L^2( \partial (t_m \Omega_m))}^2} =  \frac{||\nabla u||_{L^2(\Omega_0^\text{ext})}^2}{||u||_{L^2( \partial \Omega_0)}^2}.
\end{align*}
Since  $\mu_1( t_m \Omega_m^\text{ext} ) = \frac{\mu_1( \Omega_m^\text{ext} )}{t_m}$, we obtain 
\begin{align*}
\mu_1( \Omega_0^\text{ext} ) \geq \limsup_{m \to \infty} \mu_1(t_m \Omega_m^\text{ext} ) = \limsup_{m \to \infty} \mu_1(\Omega_m^\text{ext}).
\end{align*}

In the second step, we prove the lower semicontinuity. With $u_m \in E^1(\Omega_m^\text{ext})$,  we denote a sequence of functions satisfying $||u_m||_{L^2(\partial \Omega_m)} = 1$ and 
\begin{align*}
\mu_1(\Omega_m^\text{ext}) =  \frac{|| \nabla u_m||_{L^2(\Omega_m^\text{ext})}^2}{||  u_m||_{L^2( \partial \Omega_m)}^2}=  || \nabla u_m||_{L^2(\Omega_m^\text{ext})}^2.
\end{align*}
By the upper semicontinuity, there exists a $K>0$ such that $\mu_1(\Omega_m^\text{ext}) \leq K$. Therefore, it holds $||u_m||_{E^{1}(\Omega_m^\text{ext})}^2 \leq K + \frac{1}{|\partial \Omega_m|} \leq K+\frac{1}{|\partial B_r|}$. By \cite[Corollary 3.4]{Auchmuty2014}, there exists a constant $C_{\Omega_m}$ such that 
\begin{align*}
\frac{1}{C_{\Omega_m}} ||u_m||_{L^\frac{2n}{n-2}(\Omega_m^\text{ext})} \leq ||u_m||_{E^{1}(\Omega_m^\text{ext})} \leq \sqrt{K+\frac{1}{|\partial B_r|}}.
\end{align*}
In view of the construction of $C_{\Omega_m}$, it can be chosen independent of $m$ since $\Omega_m \subseteq B_R$. Thus, there exists a constant $K_1$ with  
\begin{align*}
||u_m||_{L^\frac{2n}{n-2}(\Omega_m^\text{ext}\cap B_R)} \leq ||u_m||_{L^\frac{2n}{n-2}(\Omega_m^\text{ext})} \leq K_1.
\end{align*}
Since $\frac{2n}{n-2} \geq 2$, the H\"older inequality leads to
\begin{align*}
||u_m||_{L^2(\Omega_m^\text{ext}\cap B_R)} \leq |\Omega_m^\text{ext}\cap B_R|^{\frac{1}{2}-\frac{n-2}{2n}} ||u_m||_{L^\frac{2n}{n-2}(\Omega_m^\text{ext}\cap B_R)} < |B_R \setminus B_r|^\frac{1}{n} K_1.
\end{align*}
Thus, there exists a constant $K_2>0$, independent of $m$ with $||u_m||_{W^{1,2}(\Omega_m^\text{ext}\cap B_R)} \leq K_2$. Analogous to Lemma \ref{lemma:extension}, there are extensions of $u_m$, denoted by $\widehat{u_m} \in W^{1,2}(B_R) $, such that $ || \widehat{u_m}||_{W^{1,2}(B_R)} \leq K_3$. Furthermore, we define the function $\widetilde{u_m} \in E^1(B_r^\text{ext})$ by
\begin{align*}
\widetilde{u_m}(x):=\begin{cases}
\widehat{u_m}(x) \, &\text{ if } x \in B_R \setminus \overline{B_r}, \\
u_m(x)\, &\text{ if } x \in \Omega_m^\text{ext}.
\end{cases}
\end{align*}
 Then, since the trace operator on $B_R \cap B_r^\text{ext}$ is bounded, we have
\begin{align*}
|| \widetilde{u_m}||_{E^1(B_r^\text{ext})} &\leq C || \widehat{u_m}||_{W^{1,2}(B_r^\text{ext} \cap B_R)} + || u_m||_{E^1(\Omega_m^\text{ext})}  \leq  C K_3 + \sqrt{K+\frac{1}{| \partial B_r |}}.
\end{align*}
Hence, there is a subsequence - again denoted by $\widetilde{u_m}$ - which is weakly convergent.  We denote the limit by $\widetilde{u} \in E^1(B_r^\text{ext}) $. Let $\chi_m$ denote the characteristic function on $\Omega_m^\text{ext}$ and $\chi$ denotes the characteristic function on $\Omega_0^\text{ext}$. For any $\phi \in C_0^\infty$, it holds
\begin{align*}
&\int_{\Omega_m^\text{ext}} | \nabla \widetilde{ u_m}| \phi \, \mathrm{d} x - \int_{\Omega_0^\text{ext}} | \nabla \widetilde{u}| \phi \, \mathrm{d} x \\
=& \int_{B_r^\text{ext}} | \nabla \widetilde{ u_m}| \chi_m \phi - | \nabla \widetilde{u}| \chi \phi \, \mathrm{d} x \\ 
\leq & \left| \int_{B_r^\text{ext}} \left( | \nabla \widetilde{ u_m}|  - | \nabla \widetilde{u}| \right) \chi_m \phi \, \mathrm{d} x \right| + \left| \int_{B_r^\text{ext}}  | \nabla \widetilde{u}| \phi  \left( \chi_m  -  \chi \right) \, \mathrm{d} x \right|.
\end{align*}
The last integral obviously vanishes for $m \to \infty$. Furthermore, $|\nabla \widetilde{ u_m}|  - | \nabla \widetilde{u}|$ converges weakly to zero and $\chi_m \phi$ converges strongly to $\chi \phi$. Thus, the first integral also vanishes for $m \to \infty$. Hence, $|  \nabla \widetilde{ u_m}| \chi_m$ converges weakly to $|\nabla \widetilde{u}| \chi$ in $L^2(B_r^\text{ext})$. In view of the lower semicontinuity of the $L^2$-norm with respect to weak convergence, we obtain
\begin{align*}
\liminf_{m \to \infty} || \nabla u_m ||_{L^2(\Omega_m^\text{ext})}^2 = \liminf_{m \to \infty} ||  \chi_m \nabla \widetilde{u_m} ||_{L^2(B_r^\text{ext})}^2 \geq  ||  \chi \nabla \widetilde{u} ||_{L^2(B_r^\text{ext})}^2 =||   \nabla \widetilde{u} ||_{L^2(\Omega_0^\text{ext})}^2.
\end{align*}
For the boundary integral, we can use \cite[(5.5)]{bucur2011first}, i.e.
\begin{align*}
\lim_{m \to \infty} \int_{ \partial \Omega_m} u_m^2 dS = \int_{\partial \Omega_0} \widetilde{u}^2 dS.
\end{align*}
Therefore, we obtain
\begin{align*}
\liminf_{ m \to \infty} \mu_1(\Omega_m^\text{ext}) \geq \frac{||   \nabla {\widetilde{u}} ||_{L^2(\Omega_0^\text{ext})}^2}{||   \widetilde{u} ||_{L^2( \partial \Omega_0)}^2} \geq \mu_1(\Omega_0^\text{ext}).
\end{align*}
\end{proof}

In a similar way, we obtain the continuity of the Robin eigenvalue.
\vspace{1ex}
\begin{theorem}
Let $\alpha < \alpha^*(\Omega_0^\text{ext}) = - \mu_1(\Omega_0^\text{ext})$ for a fixed $\Omega_0 \in X$. The mapping
\begin{align*}
f_2: X \to \R, \, \Omega \mapsto \lambda_1^\alpha(\Omega^{\text{ext}})
\end{align*}
is continuous with respect to the Hausdorff metric at $\Omega_0$.
\end{theorem}
\begin{proof}
Let $(\Omega_m)_{m \in \mathbb{N}} \subset X$ be a sequence as in the proof of Lemma \ref{lemma:contf1}. Since $f_1$ is continuous with respect to the Hausdorff metric, the condition $\alpha < \alpha^*(\Omega_0^{\text{ext}})$ implies $\alpha < \alpha^*(\Omega_m^{\text{ext}})$ for all sufficiently large $m$. Without restriction, we can assume that $\alpha < \alpha^*(\Omega_m^{\text{ext}})$ holds for all $m \in \N_0$.

We start by proving the upper semicontinuity. To this end, we consider a sequence $(t_m)_{m \in \N} \subset (1, \infty)$ with $\lim_{m \to \infty} t_m = 1$ and $ \overline{\Omega_0} \subseteq t_m \Omega_m $. If $u_0$ denotes the eigenfunction corresponding to $\lambda_1^\alpha(\Omega_0^\text{ext})$, we obtain, analogously to Lemma \ref{lemma:contf1},
\begin{align*}
\lambda_1^\alpha(\Omega_0^\text{ext})  = \lim_{m \to \infty} \frac{\int_{t_m \Omega_m^\text{ext}} | \nabla u_0 |^2 \, \mathrm{d} x + \alpha \int_{\partial (t_m \Omega_m)} u_0^2 dS}{\int_{t_m \Omega_m^\text{ext}}  u_0^2 \, \mathrm{d} x} \geq \limsup_{m \to \infty} \lambda_1^\alpha( t_m \Omega_m^\text{ext}).
\end{align*}
With $\lambda_1^\frac{\alpha}{t_m}(t_m \Omega^\text{ext}) = \frac{1}{t_m^2} \lambda_1^\alpha( \Omega^\text{ext})$, and the continuity of $\alpha \mapsto \lambda_1^\alpha(\Omega^\text{ext})$, we obtain
\begin{align*}
\lambda_1^\alpha(\Omega_0^\text{ext}) \geq \limsup_{m \to \infty} \lambda_1^\alpha( t_m \Omega_m^\text{ext}) = \limsup_{m \to \infty} \frac{1}{t_m^2} \lambda_1^{ t_m \alpha}( \Omega_m^\text{ext}) = \limsup_{m \to \infty} \lambda_1^\alpha( \Omega_m^\text{ext}).
\end{align*}
It remains to show the lower semicontinuity. Let $u_m \in W^{1,2}(\Omega_m^\text{ext})$ be the eigenfunction corresponding to $\lambda_1^\alpha(\Omega_m^\text{ext})$ with $||u_m||_{L^2(\partial \Omega_m)}=1$. Due to the upper semicontinuity, there exists a $K_1 > 0$ such that
\begin{align*}
-K_1 \geq \frac{|| \nabla u_m||_{L^2(\Omega_m^\text{ext})}^2 + \alpha }{|| u_m||_{L^2(\Omega_m^\text{ext})}^2} \geq \frac{\mu_1(\Omega_m^\text{ext}) + \alpha }{|| u_m||_{L^2(\Omega_m^\text{ext})}^2},
\end{align*}
where we used \eqref{eq:varcharstek} for the second step. Since we assumed $\alpha < \alpha^*(\Omega_m^\text{ext})$, it holds $ \mu_1(\Omega_m^\text{ext}) + \alpha  <0$. Thus, it holds
\begin{align*}
\lim_{m \to \infty} || u_m||_{L^2(\Omega_m^\text{ext})}^2 <\lim_{m \to \infty} \frac{|\mu_1(\Omega_m^\text{ext}) + \alpha |}{2 K_1} = \frac{|\mu_1(\Omega_0^\text{ext}) + \alpha |}{2 K_1}.
\end{align*}
Furthermore, $|| \nabla u_m||_{L^2(\Omega_m^\text{ext})}^2 < - \alpha$ because $\lambda_1^\alpha(\Omega_m^\text{ext})<0$. Hence, there is a constant $K_2>0$ with $||u_m||_{W^{1,2}(\Omega_m^\text{ext})} \leq K_2$.
In view of Lemma \ref{lemma:extension}, there are extensions $\widetilde{u_m} \in W^{1,2}(B_r^\text{ext})$ of $u_m$ with $|| u_m ||_{W^{1,2}(B_r^\text{ext})} \leq K$. Thus, there is a weakly convergent subsequence with limit $\widetilde{u} \in W^{1,2}(B_r^\text{ext})$. Analogous to Lemma \ref{lemma:contf1}, it holds
\begin{align*}
\liminf_{m \to \infty} ||  \nabla u_m ||_{L^2(\Omega_m^\text{ext})}^2 = \liminf_{m \to \infty} || | \chi_m \nabla \widetilde{u_m} ||_{L^2(B_r^\text{ext})}^2 \geq  ||  \chi \nabla \widetilde{u} ||_{L^2(B_r^\text{ext})}^2 =||   \nabla \widetilde{u} ||_{L^2(\Omega_0^\text{ext})}^2
\end{align*}
and
\begin{align*}
\int_{\Omega_m^\text{ext}}  \widetilde{ u_m}  \phi \, \mathrm{d} x - \int_{\Omega_0^\text{ext}}    \widetilde{u}  \phi \, \mathrm{d} x =& \int_{B_r^\text{ext}}    \widetilde{ u_m}  \chi_m \phi -    \widetilde{u}  \chi \phi \, \mathrm{d} x \\ 
\leq & \left| \int_{B_r^\text{ext}} \left(     \widetilde{ u_m}   -     \widetilde{u}  \right) \chi_m \phi \, \mathrm{d} x \right| + \left|  \int_{B_r^\text{ext}}      \widetilde{u}  \phi  \left( \chi_m  -  \chi \right) \, \mathrm{d} x \right|.
\end{align*}
Analogous to Lemma \ref{lemma:contf1}, we obtain
\begin{align*}
\liminf_{m \to \infty} ||  u_m ||_{L^2(\Omega_m^\text{ext})}^2 = \liminf_{m \to \infty} ||  \chi_m  \widetilde{u_m} ||_{L^2(B_r^\text{ext})}^2 \geq  ||  \chi  \widetilde{u} ||_{L^2(B_r^\text{ext})}^2 =||   \widetilde{u} ||_{L^2(\Omega_0^\text{ext})}^2.
\end{align*}
In total, considering $|| \nabla u_m ||_{L^2(\Omega_m^\text{ext})}^2 + \alpha ||  u_m ||_{L^2( \partial \Omega_m)}^2 <0$, we obtain
\begin{align*}
\liminf_{m \to \infty} \lambda_1^\alpha(\Omega_m^\text{ext})  \geq& \, \liminf_{m \to \infty} \frac{|| \nabla \widetilde{u} ||_{L^2(\Omega_0^\text{ext})}^2 + \alpha ||  u_m ||_{L^2( \partial \Omega_m)}^2 }{||  \widetilde{u} ||_{L^2(\Omega_0^\text{ext})}^2} \\
=& \, \frac{|| \nabla \widetilde{u} ||_{L^2(\Omega_0^\text{ext})}^2 + \alpha ||  \widetilde{u} ||_{L^2( \partial \Omega_0)}^2 }{||  \widetilde{u} ||_{L^2(\Omega_0^\text{ext})}^2} \geq \lambda_1^\alpha(\Omega_0^\text{ext}).
\end{align*}
Thus, the lower and the upper semicontinuity of $f_2$ are shown.
\end{proof}

The continuity of $f_1$ and $f_2$ allows us to conclude the existence of an optimal domain. However, in Definition \ref{def:HausX}, we imposed the condition $\Omega \subseteq D$, where $D$ is a bounded domain. To deduce the existence of a global maximizer, one must ensure that a minimizing sequence cannot become arbitrarily thin. Currently, we are unable to guarantee this and the example in Section \ref{sec:glob} indicates that this could be problematic.

\vspace{1ex}
\begin{korollar}
Among all domains $\Omega \in X$ with the same measure as the unit ball, there exists a domain $\Omega_0$, that maximizes $\mu_1 (\Omega^\text{ext})$.

For $\alpha < \alpha^*(\Omega_0^\text{ext}) = - \mu_1 (\Omega_0^\text{ext})$, there is a domain $\Omega_1$, maximizing $\lambda_1^\alpha(\Omega^\text{ext})$ among all domains $X$ with the same measure as the unit ball.
\end{korollar}
\begin{proof}
We set $\widehat{X}:= \{ \Omega \in X: | \Omega | = |B_1 | \}$. Let $ \left( \Omega_m \right)_{m \in \mathbb{N}} \subset \widehat{X}$ be a sequence such that
\begin{align*}
\lim_{m \to \infty} \mu_1 (\Omega_m^\text{ext}) = \sup_{\Omega \in \widehat{X}} \mu_1 (\Omega^\text{ext}). 
\end{align*}
Since $\Omega_m \subseteq D$, the Blaschke selection theorem gives the existence of a converging subsequence (in the Hausdorff metric). Thus, the supremum is attained.

Now, let $\alpha < \alpha^*(\Omega_0^\text{ext})$ and let $ \left( \Omega_m \right)_{m \in \mathbb{N}} \subset \widehat{X}$ be a sequence such that
\begin{align*}
\lim_{m \to \infty} \lambda_1^\alpha (\Omega_m^\text{ext}) = \sup_{\Omega \in \widehat{X}} \lambda_1^\alpha (\Omega^\text{ext}). 
\end{align*}
Since $\alpha < \alpha^*(\Omega_0^\text{ext})$, we have continuity. Again, the Blaschke selection theorem gives the existence of a converging subsequence. Thus, the supremum is attained.
\end{proof}

\subsection{Monotonicity}\label{sec:mon}
In this section, we present a monotonicity result with respect to a certain kind of domain inclusion. We proceed analogues to \cite{giorgi2005monotonicity}, where T. Giorgi and R.G. Smits give a similar result for the first Robin eigenvalue on bounded domains. 
\vspace{1ex}
\begin{theorem}\label{theo:tiz}
Let $\Omega \subseteq \mathbb{R}^n$ be a Lipschitz domain and $B_r \subseteq \Omega$. Then, 
 \begin{align*}
 \mu_1(\Omega^{\text{ext}}) < \mu_1(B_r^{\text{ext}}) \, \text{ for } n \geq 3.
\end{align*}
For $n = 2$ and $\alpha <0$ or $n \geq 3$ and  $\alpha < -\mu_1(B_r^{\text{ext}})$, it holds
\begin{align*}
\lambda_1^\alpha(\Omega^{\text{ext}}) \leq \lambda_1^\alpha(B_r^{\text{ext}}).
\end{align*}
\end{theorem}
\begin{proof}
Let $\phi$ denote the eigenfunction corresponding to $\lambda_1^\alpha(B_r^{\text{ext}})$. The function $\phi$ is radial and we write $\phi(x) = f(|x|)$. We define the function
\begin{align*}
\widehat{\alpha}: \partial \Omega \to \R, \, \, y \mapsto \frac{\langle \nabla \phi(y), \nu(y) \rangle}{\phi(y)},
\end{align*}
where $\nu(y)$ denotes the outer normal on $\partial \Omega$. Since $f'(|y|) < 0$ and $\langle \frac{y}{|y|} , \nu \rangle < 1$, it holds
\begin{align*}
\langle \nabla \phi(y), \nu \rangle = \langle \frac{y}{|y|} f'(|y|), \nu \rangle = f'(|y|) \langle \frac{y}{|y|} , \nu \rangle > f'(|y|).
\end{align*}
Thus, $\widehat{\alpha}(y) > \frac{f'(|y|)}{f(|y|)}$. In view of \eqref{eq:mondecf}, $- \frac{f'(x)}{f(x)}$ is monotonically decreasing, which implies
\begin{align*}
\widehat{\alpha}(y) > \frac{f'(|y|)}{f(|y|)} > \frac{f'(r)}{f(r)} = \alpha.
\end{align*}
Furthermore, $\phi$ solves the equation
\begin{align*}
\begin{cases}
\Delta \phi + \lambda_1^\alpha(B_r^{\text{ext}}) \phi = 0 &\text{ in } \Omega^{\text{ext}} \subseteq B_r^{\text{ext}}, \\
- \partial_\nu \phi + \widehat{\alpha} \phi = 0 &\text{ on } \partial \Omega.
\end{cases}
\end{align*}
Thus, integration by parts yields
\begin{align*}
\lambda_1^\alpha(B_r^{\text{ext}}) &=  \frac{\int_{\Omega^{\text{ext}}} | \nabla \phi |^2 \, \mathrm{d} x +  \int_{\partial \Omega} \widehat{\alpha} \phi^2 \, \mathrm{d} S }{\int_{\Omega^{\text{ext}}}  \phi^2 \, \mathrm{d} x} >\frac{\int_{\Omega^{\text{ext}}} | \nabla \phi |^2 \, \mathrm{d} x +  \int_{\partial \Omega} \alpha  \phi^2 \, \mathrm{d}S }{\int_{\Omega^{\text{ext}}}  \phi^2 \, \mathrm{d}x} \\
& \geq \inf_{u \in W^{1,2}(\Omega^{\text{ext}})} \frac{\int_{\Omega^{\text{ext}}} | \nabla u |^2 \, \mathrm{d}x +  \alpha \int_{\partial \Omega}   u^2  \,\mathrm{d}S }{\int_{\Omega^{\text{ext}}}  u^2 \, \mathrm{d}x} = \lambda_1^\alpha(\Omega^{\text{ext}}).
\end{align*}
The inequality for the Steklov eigenvalue is a consequence of  $\lambda_1^\alpha(\Omega^{\text{ext}}) \leq \lambda_1^\alpha(B_r^{\text{ext}})$ and  $\lambda_1^\alpha(\Omega^{\text{ext}}) < 0 \Leftrightarrow \alpha < - \mu_1(\Omega^{\text{ext}})$.
\end{proof}

However, monotonicity with respect to domain inclusions does not hold true in general, even if we consider convex domains, as demonstrated in the following example.

\vspace{1ex}
\begin{remark}
For $a > 0$, let $S_a:= (a,a)^2 \in \mathbb{R}^2$. For $\alpha \to - \infty$, it holds
\begin{align*}
\lambda_1^\alpha(S_1^\text{ext}) < \lambda_1^\alpha(B_1^\text{ext}) \, \, \text{ and } \, \, \lambda_1^\alpha\left( S_{\frac{1}{\sqrt{2}}}^\text{ext} \right) < \lambda_1^\alpha(B_1^\text{ext}). 
\end{align*}
\end{remark}
\begin{proof}
In \cite[Theorem 1.1]{HYNEK},  K. Kova{\v{r}}{\'\i}k, and K. Pankrashkin show an asymptotic behavior for the Robin eigenvalue, which yields
\begin{align}\label{eq:KoPa}
\lambda_1^\alpha(\Omega^\text{{ext}}) = -\alpha^2 + (n-1) H_{\text{{max}}}(\Omega^{\text{{ext}}}) \alpha + o(\alpha) \, \, \text{ as } \, \, \alpha \to - \infty
\end{align}
for domains $\Omega \subseteq \mathbb{R}^n$ with $\partial \Omega \in \mathcal{C}^{1,1}$. Here, $H_{\text{max}}$ denotes the maximal mean curvature. Note that $H(\Omega^\text{ext}) = - H(\Omega)$, i.e. $H_\text{max}(\Omega^\text{ext})$ is the maximal (minimal in absolute value) curvature of $\partial \Omega^\text{ext}$. Thus, it holds
\begin{align*}
H_\text{max}(B_1^\text{ext}) = -1 \, \, \text{ and } \, \, H_\text{max}(S_a^\text{ext}) = 0 \, \text{ for all } a > 0.
\end{align*}
Based on the results from Section \ref{sec:cont}, we can approximate $S_a$ by smooth domains. Therefore, we obtain
\begin{align*}
H_\text{max}(B_1^\text{ext}) = - \alpha^2 - \alpha + o(\alpha)  \, \, \text{ and } \, \, H_\text{max}(S_a^\text{ext}) = - \alpha^2 + o(\alpha).
\end{align*}
Since $S_{\frac{1}{\sqrt{2}}} \subseteq B_1 \subseteq S_1$, this shows that in general there can be no monotonicity with respect to domain inclusion.
\end{proof}

\subsection{Global Optimization}\label{sec:glob}
In dimension $n=2$, the ball is the global maximizer of $\lambda_1^\alpha(\Omega^\text{ext})$ among all smooth, bounded, simply connected sets $\Omega$ with given perimeter or given area. In higher dimensions, this is not true anymore. D. Krejcirik and V. Lotoreichik  construct in \cite[Section 5]{krejcirik2016optimisation}  a convex domain $\Omega$ with $\partial \Omega \in \mathcal{C}^{1,1}$ and $|B_R| = |\Omega|$ such that 
\begin{align*}
 \lambda_1^\alpha ( \Omega^\text{ext}) >  \lambda_1^\alpha ( B_R^\text{ext}) \, \, \text{ as } \, \, \alpha \to - \infty.
\end{align*}
Following this idea, we construct for $n \geq 3$ a convex domain $E \subseteq \mathbb{R}^n$ with $|E| = |B_R|$ and $\partial E \in \mathcal{C}^\infty$ such that $ \lambda_1^\alpha ( E^\text{ext}) >  \lambda_1^\alpha ( B_R^\text{ext})$ for $\alpha$ sufficiently negative. For $a \in (0,1) $ and $n \geq 3$, we consider the hyperellipsoid
\begin{align*}
E(a):= \left\{ x \in \R^n: (a x_1)^2 + \sum_{k=2}^n x_k^2 \leq 1 \right\}.
\end{align*}
Due to \cite[Theorem B]{Leung}, the principal curvatures of $\partial E(a)$  are given by
\begin{align*}
\kappa_1 = \ldots = \kappa_{n-2} = \frac{1}{\sqrt{1+a^2(a^2-1)x_1^2}} \, \, \text{ and } \, \, \kappa_{n-1} =\frac{a^2}{\left( 1+ a^2(a^2-1) x_1^2 \right)^\frac{3}{2}}.
\end{align*}
Therefore, the mean curvature of $\partial E(a)^\text{ext}$ equals
\begin{align*}
H(x) = -\frac{\frac{n-2}{\sqrt{1+a^2(a^2-1)x_1^2}}+\frac{a^2}{\left( 1+ a^2(a^2-1) x_1^2 \right)^\frac{3}{2}}}{n-1}.
\end{align*}
Since $a^2-1 < 0$, this becomes maximal (minimal in absolute value) for $x_1 = 0$ with
\begin{align*}
H_{\text{max}}(E(a)^\text{ext}) = -\frac{n-2+a^2}{n-1}.
\end{align*}
It holds that $|E(a)| = \frac{1}{a} |B_1| =  \left| B_{\frac{1}{\sqrt[n]{a}}} \right|$. Thus, we want to compare $\lambda_1^\alpha \left( B^\text{ext} \right)$ and $\lambda_1^\alpha ( E(a)^\text{ext})$, where $B:=B_{\frac{1}{\sqrt[n]{a}}}$. Using \eqref{eq:KoPa}, we obtain
\begin{align*}
\lambda_1^\alpha ( E(a)^\text{ext}) &= - \alpha^2 - (n-2+a^2) \alpha + o(\alpha),\\
\lambda_1^\alpha ( B^\text{ext}) &= - \alpha^2 - \alpha (n-1)\sqrt[n]{a} + o(\alpha).
\end{align*} 
For small $a$ and $n \geq 3$, it holds  
\begin{align}\label{eq:Hynak}
n-2+a^2 > (n-1)\sqrt[n]{a}.
\end{align}
Thus, it holds that $\lambda_1^\alpha ( E(a)^\text{ext}) > \lambda_1^\alpha \left( B^\text{ext} \right)$ for $\alpha$ sufficiently negative. Thus, the ball is not in general the maximizer of the first Robin eigenvalue among all smooth convex domains with given measure.

If the ball maximized $\lambda_1^\alpha (\Omega^\text{ext})$ among all smooth domains with given perimeter, every domain $\Omega$ with $| \partial \Omega| = |\partial B_{R_1}|$ would fulfill $\lambda_1^\alpha (\Omega^\text{ext}) < \lambda_1^\alpha (B_{R_1}^\text{ext})$. Let $B_{R_2}$ be the ball such that $| \Omega| = |B_{R_2}|$. The classic isoperimetric inequality yields $R_2 \leq R_1$. Thus, Theorem \ref{theo:tiz} yields $\lambda_1^\alpha(R_1^\text{ext}) < \lambda_1^\alpha(R_2^\text{ext})$ and we obtain 
\begin{align*}
\lambda_1^\alpha(\Omega^\text{ext}) < \lambda_1^\alpha(R_1^\text{ext}) < \lambda_1^\alpha(R_2^\text{ext}),
\end{align*}
which means that the ball also maximizes $\lambda_1^\alpha (\Omega^\text{ext})$ among all smooth domains with given measure. So the ball is neither the global maximizer of the first Robin eigenvalue under measure nor under perimeter constrains for all $\alpha$.

 However, \eqref{eq:Hynak} is not true for $a$ close to $1$, i.e. $E(a)$ close to $B_1$. Hence, the ball might be a local maximizer of $\lambda_1^\alpha(\Omega^\text{ext})$.


\section{Domain Variations in Exterior Domains}\label{sec:DomVat}
In the absence of a global approach, our study focuses on the effect of local perturbations on the eigenvalue. For bounded domains, the concept of domain variations is explained in \cite{Henrot} by A. Henrot and M. Pierre. Similarly, C. Bandle and A. Wagner apply this approach in \cite{Wagner1}, demonstrating its applicability to a range of problems. In this paper, we proceed in a similar way. Notably, some of the results presented in Section \ref{sec:DomVat} have already been published in \cite{BundPhD}.
\subsection{Basic Concepts}\label{sec:DomVarInt}
Let $\Omega \subseteq \mathbb{R}^n$ be a connected, bounded, and smooth domain. For a fixed $t_0 > 0$, let $\Phi: (-t_0,t_0) \times \mathbb{R}^n \setminus \Omega \to  \mathbb{R}^n$ be smooth in $t$ and $x$. Then, we call $\left( \Omega_t^\text{ext} \right)_{|t| < t_0}$, given by
\begin{align*}
\Omega_t^{\text{ext}} := \Phi(t,\Omega^{\text{ext}}),
\end{align*}
a  \textit{family of perturbations} of $\Omega^{\text{ext}}$. If $\Omega^{\text{ext}}$ maximizes a functional $\mathcal{G}: A \to \mathbb{R}$, where $A$ is a suitable collection of subsets of $\mathbb{R}^n$ and $\Omega_t^\text{ext} \in A$ for all $|t| < t_0$, then the mapping $t \mapsto \mathcal{G}(\Omega_t^\text{ext})$ must have a maximum at $t=0$. 

To apply this to $\lambda_1^\alpha (\Omega)$, we assume that $ \partial \Omega \in \mathcal{C}^{2, \beta}$, $0 < \beta \leq 1$, and $\Phi$ is of the form
\begin{align}\label{eq:Phiext}
\Phi(t,x) = x + t v(x) + \frac{t^2}{2} w(x) + f(t,x),
\end{align}
where $v,w, f(t, \cdot) \in \mathcal{C}^{2, \beta}(\R^n \setminus \Omega)$ and $f(t,x) = o(t^2)$ as $t \to 0$ . In addition, we assume
\vspace{1ex}
\begin{enumerate}[(a)]
\item the mapping $\Phi(t,\cdot) - id$ has compact support, \label{enum:1}
\item $ \left| \left| \frac{\Phi(t,\cdot)-id}{t} \right| \right|_{\mathcal{C}^{1}(\Omega^\text{ext})} \leq C(\Omega)$ for all $|t| < t_0$ and $C(\Omega)$ is defined in Lemma \ref{lemma:meanvalext}, \label{enum:ot1}
\item $\Phi$ is a \textit{Hadamard perturbation}, i.e. $v = \langle v, \nu \rangle \nu$ and $w = \langle w, \nu \rangle \nu$. \label{enum:perturexta}
\end{enumerate}
\vspace{1ex}
The Jacobian matrix of $\Phi(t, \cdot)$, i.e. $ \left( D{\Phi(t, x)} \right)_{i,j} := \frac{\partial \Phi(t, x)_i}{\partial x_j}$, equals
\begin{align*}
D {\Phi(t, x)} = I + t Dv(x) + \frac{t^2}{2} Dw(x) + D f(t,x).
\end{align*}
Applying Jacobi's formula for determinants, as in \cite[Section A.1]{Wagner1}, results in
\begin{align*}
J(t,x) :=&\, \operatorname{det} \left( D {\Phi(t, x)} \right) \\
=& \,1 + t \operatorname{div}(v(x)) \nonumber + \frac{t^2\left[  \operatorname{div}(v(x))^2 + \operatorname{div}(w(x)) - \operatorname{tr} \left(  (D v(x))^2 \right) \right]}{2}  + g(t,x),
\end{align*}
where $g(t,x) = o(t^2)$. Given the condition \eqref{enum:ot1}, there exists a constant $c>0$ that remains independent of the specific perturbation, such that $ \left| g(t,x) \right| \leq c |t|$. Consequently, we can choose $t_0$ such that $J(t,x) > 0$ for all $|t| < t_0$ and for all $x \in \Omega^{\text{ext}}$. Thus, $\Phi(t, \cdot) : \Omega^{\text{ext}} \to \Omega_t^{\text{ext}}$ is a local diffeomorphism. To show that $\Phi(t, \cdot)$ is a global diffeomorphism, it remains to show that it is bijective.

\vspace{1ex}
\begin{lemma}\label{lemma:meanvalext}
Let $\operatorname{geo}:\partial \Omega \times \partial \Omega \to \mathbb{R}$ denote the geodesic distance between to points on the boundary. We define the quantity 
\begin{align*}
L(\Omega) := \sup_{x, y \in \partial \Omega, x \neq y} \frac{\operatorname{geo}(x,y)}{|x-y|}.
\end{align*}
If we choose $C(\Omega) < \frac{1}{L(\Omega) t_0}$ in condition \eqref{enum:ot1}, then $\Phi(t, \cdot) : \Omega^{\text{ext}} \to \Omega_t^{\text{ext}}$ is bijective. 
\end{lemma}
\begin{proof}
We start by assuming that for a fixed $|t| < t_0$, there exist $x_1, x_2 \in \Omega^\text{ext}$ such that $\Phi(t,x_1) = \Phi(t,x_2)$ and $x_1 \neq x_2$. In addition, we define $\widehat{\Phi}(t,x):= \frac{\Phi(t,x)-x}{t}$. Then, 
\begin{align*}
|x_1 - x_2| = \left| t \frac{x_1 - x_2 + \Phi(t,x_2) - \Phi(t,x_1)}{t} \right| = |t| \left| \widehat{\Phi}(t,x_1) - \widehat{\Phi}(t,x_2) \right|.
\end{align*}
We can construct a piecewise continuously differentiable curve $\psi: [0,1] \to \mathbb{R} \setminus \Omega$ such that $\psi(0) = x_1,$ $\psi(1) = x_2$ and $ |\psi '(s)|$ is constant. Specifically, this curve can be chosen such that  $| \psi ' (s) |  \leq L(\Omega) |x_1 - x_2|$. With $g:[0,1] \to \mathbb{R}^n, \, g(s) := \widehat{\Phi}(t,{\psi}(s))$ ,we have 
\begin{align*}
|x_1 - x_2| &= |t| \left| \widehat{\Phi}(t,x_1) - \widehat{\Phi}(t,x_2) \right| = |t| \left| g(0) - g(1) \right| =|t| \left| \int_0^1 g'(s) \, \mathrm{d}s \right| \\
&< t_0 \int_0^1 || D \widehat{\Phi}(t, \cdot)||_{L^\infty(\Omega^\text{ext})} | \psi'(s)|  \, \mathrm{d}s < t_0 C(\Omega) L(\Omega) |x_2 - x_1|.
\end{align*}
Thus, the choice of $C(\Omega) < \frac{1}{L(\Omega) t_0}$ contradicts $x_1 \neq x_2$.
\end{proof}

Hence, for $f \in L^1(\Omega_t^{\text{ext}})$, it holds
\begin{align*}
\int_{\Omega_t^{\text{ext}}} f(x) \, \mathrm{d}x = \int_{\Omega^{\text{ext}}} f \left( \Phi(t,x) \right) J(t,x) \, \mathrm{d}x   .
\end{align*}
This formula is a direct consequence of the transformation formula applied to bounded domains and condition \eqref{enum:1}. Similarly, we can transform boundary integrals. Assuming $\partial \Omega$ is smooth, it can be represented by local coordinates, i.e. let $V \subseteq \R^n$ be an open set such that $V \cap \partial \Omega = \{ \widetilde{x}(\xi) : \xi \in U \subseteq \R^{n-1} \}$.
Define the matrix $G\in \mathbb{R}^{(n-1) \times (n-1)}$ by $G_{i,j}  := \langle \widetilde{x}_{\xi_i},\widetilde{x}_{\xi_j} \rangle$, where $ \widetilde{x}_{\xi_i}$ is the derivative of $\widetilde{x}(\xi)$ with respect to $ \xi_i$. Then, the surface element is given by $\mathrm{d}S = \sqrt{\operatorname{det}(G)} \, \mathrm{d} \xi$, i.e.
\begin{align*}
\int_{V \cap \partial \Omega} f(x) \,  \mathrm{d}S = \int_U f(\widetilde{x}(\xi)) \sqrt{\operatorname{det}(G(\xi))} \, \mathrm{d} \xi 
\end{align*}
holds for any continuous function $f$. There is a finite covering of open sets $\{ V_i \}_{i=1}^m$ and a partition of unity $\{ p_i \}_{i=1}^m$ such that
\begin{align*}
\int_{\partial \Omega} f(x) \,  \mathrm{d}S = \sum_{i=1}^m \int_{V_i \cap \partial \Omega} f(x) p_i(x) \,  \mathrm{d}S.
\end{align*}
This yields a formula for the surface element on the entire boundary. The boundary of $\Omega_t$ can locally be written as
\begin{align*}
 \left\{ \widetilde{y}(\xi) = \widetilde{x}(\xi) + t \widetilde{v}(\xi) + \frac{t^2}{2} \widetilde{w}(\xi) + o(t^2) : \xi \in U \subseteq \R^{n-1} \right\}.
\end{align*}
Thus, the surface element on $\partial \Omega_t$ equals $\mathrm{d} S_t = \sqrt{\operatorname{det}(K)} \, \mathrm{d} \xi$, where $K \in \mathbb{R}^{(n-1) \times (n-1)}$ is defined as $K_{i,j} = \langle \widetilde{y}_{\xi_i},  \widetilde{y}_{\xi_j} \rangle$. To express $\mathrm{d} S_t$ in the form $\mathrm{d} S_t = m(t,x) \,  \mathrm{d} S$, we introduce the matrices $A, B \in \mathbb{R}^{(n-1) \times (n-1)}$, given by
\begin{align*}
A_{i,j} :=& \langle \widetilde{x}_{\xi_i},  \widetilde{v}_{\xi_j} \rangle + \langle \widetilde{x}_{\xi_j},  \widetilde{v}_{\xi_i} \rangle, \\
B_{i,j} :=& 2 \langle \widetilde{v}_{\xi_i},  \widetilde{v}_{\xi_j} \rangle + \langle \widetilde{w}_{\xi_i},  \widetilde{x}_{\xi_j} \rangle + \langle \widetilde{x}_{\xi_i},  \widetilde{w}_{\xi_j} \rangle.
\end{align*}
Then, $K=G+tA + \frac{t^2}{2} B +o(t^2)= G \left( I + tG^{-1}A + \frac{t^2}{2} G^{-1}B \right) + o(t^2)$. This implies
\begin{align}\label{eq:surfaceelement}
m(t,x) = \frac{\sqrt{\det(K)}}{\sqrt{\det(G)}} = \left[ \operatorname{det} \left( I + tG^{-1}A + \frac{t^2}{2} G^{-1}B \right) \right]^\frac{1}{2} + o(t^2).
\end{align}
Simplifications of $m(t,x)$ can be found in \cite{Wagner1}. Thus, for $f \in \mathcal{C}^0(\partial \Omega_t)$, we obtain
\begin{align*}
\int_{\partial \Omega_t} f(x) \,  \mathrm{d}S_t = \int_{\partial \Omega} f \left(   \Phi(t,x) \right) m(t,x)\,  \mathrm{d}S.
\end{align*}
Condition \eqref{enum:perturexta} imposes no restrictions, as any perturbation can be expressed as a Hadamard perturbation by a reparameterization of the boundary, as demonstrated in \cite[Theorem 2.1]{Wagner1}. Moreover, for $f \in W^{1,2}(\Omega^{\text{ext}})$, $g \in W^{1,2}(\Omega^{\text{ext}};\R^n)$, it holds
\begin{align*}
\int_{\Omega^{\text{ext}}} f \operatorname{div}(g) \, \mathrm{d}x = \int_{ \partial \Omega} f \langle g, -\nu \rangle\,  \mathrm{d}S - \int_{\Omega^{\text{ext}}} \langle g, \nabla f \rangle \, \mathrm{d}x,
\end{align*}
where $\nu$ denotes the outward-pointing normal to $\Omega$, meaning $-\nu$ represents the outer unit normal of the exterior domain.

The derivative of a function $u$ with respect to $t$, called the \textit{first variation}, is denoted by $\dot{u}$. We define the volume or measure of $\Omega_t$ as $V(t) :=  \left| \mathbb{R}^n \setminus \Omega_t^\text{ext}  \right|$. A perturbation satisfying $\dot{V}(0)= 0$, is called \textit{measure preserving of first order}. Using the divergence theorem and considering $\dot{J}(0,x) = \operatorname{div}(v(x))$, this can be characterized as 
\begin{align*}
\dot{V}(0) = 0 \Leftrightarrow \int_{\partial \Omega} \langle v, \nu \rangle \, \mathrm{d}S = 0.
\end{align*}
If additionally $\ddot{V}(0) = 0$, the perturbation is called \textit{measure preserving of second order}. The \textit{barycenter} of a bounded domain $\Omega \subseteq \mathbb{R}^n$ is defined as $\frac{1}{| \Omega|}\int_{\Omega} x \, \mathrm{d}x$. A perturbation satisfies the \textit{barycenter condition} if the barycenter is unchanged in first order. As in \cite[Definition 2.5]{Wagner1}, this condition is, for measure preserving perturbations, equivalent to
\begin{align*}
\int_{\partial \Omega} \langle v , \nu \rangle x \, \mathrm{d}S = 0.
\end{align*}
We often consider a family of perturbations $\left( \Omega_t \right)_{|t| < t_0}$, and a family of functions $\left( u(y,t) \right)_{|t| < t_0}$, where $u( \cdot , t): \Omega_t^\text{ext} \to \mathbb{R}$. Since the domain of $u$ depends on $t$, it is useful to consider $\widetilde{u}(x,t) := u( \Phi(t,x), t) : \Omega^\text{ext} \to \mathbb{R}$. Assuming $\widetilde{u}$ is smooth in $t$,  it holds
\begin{align*}
\dot{\widetilde{u}}(x,t) = \frac{\mathrm{d}}{\mathrm{d}t} u(\Phi(t,x),t) = \langle \nabla_y u(y,t) \big|_{y=\Phi(t,x)} , \partial_t \Phi(t,x)  \rangle + \partial_t u(y,t) \big|_{y = \Phi(t,x)}.
\end{align*}
We define the \textit{shape derivative} as $u'(x) := \partial_t u(x,t) \big|_{t=0}$ and obtain
\begin{align}\label{eq:defshape}
\dot{\widetilde{u}}(x,0) = \langle \nabla u(x,0), v(x) \rangle + u'(x).
\end{align}

\subsection{Domain Variation for the Robin Eigenvalue}\label{subsubsec:Eig}
To analyze the behavior of the first Robin eigenvalue in exterior domains, we apply the methods outlined in  Section \ref{sec:DomVarInt}. To this end, consider a family of perturbations $\Omega_t^{\text{ext}}:= \Phi(t,\Omega^{\text{ext}})$, satisfying conditions \eqref{enum:1}-\eqref{enum:perturexta}, and let $\alpha < \alpha^*(\Omega^{\text{ext}})$. In view of  Lemma \ref{lemma:contf1}, we can choose $t_0$ such that $\alpha < \alpha^*(\Omega_t^{\text{ext}})$ holds for all $|t|< t_0$. Consequently,
\begin{align}\label{eq:problemgest}
\begin{cases}
\Delta_y {u}(y,t) + \lambda {u}(y,t) = 0  & \, \text{ in } \Omega_t^{\text{ext}}, \\
-\partial_{\nu_t} {u}(y,t) + \alpha {u}(y,t) = 0 & \,  \text{ on } \partial \Omega_t,
\end{cases}
\end{align}
has a simple, negative eigenvalue $\lambda_1^\alpha(t)$ for all $|t| < t_0$. We proceed as in \cite[Section 4]{Wagner1} to transform \eqref{eq:problemgest} to a problem on $\Omega^\text{ext}$, i.e. transforming the equation
\begin{align*}
\lambda_1^\alpha (t) \int_{\Omega_t^\text{ext}} u(y,t)^2 \, \mathrm{d}y =\int_{\partial \Omega_t} \alpha u(y,t)^2\,  \mathrm{d}S_t + \int_{\Omega_t^\text{ext}} | \nabla u(y,t) |^2 \, \mathrm{d} y
\end{align*}
into an equation involving integrals over $\Omega^\text{ext}$ and $\partial \Omega$.   By introducing the transformations $y = \Phi(t,x)$ and $\widetilde{u}(x,t) = u( \Phi(t,x) , t)$ we, immediately obtain
\begin{align*}
\int_{\Omega_t^\text{ext}} u(y,t)^2 \, \mathrm{d}y &= \int_{\Omega^\text{ext}} \widetilde{u}(x,t)^2 J(t,x) \, \mathrm{d}x, \\
\int_{\partial \Omega_t} u(y,t)^2\,  \mathrm{d}S_t &= \int_{\partial \Omega} \widetilde{u}(x,t)^2 m(t,x)\,  \mathrm{d}S.
\end{align*}
For the last integral, we use the chain rule, where we write $\partial_{x_i}$ short for $\frac{\partial}{\partial x_i}$. Denoting $\Psi_t$ as the inverse of $\Phi(t, \cdot)$, and using $\frac{\partial x_i}{\partial y_k} =\partial_{y_k} (\Psi_t)_{i} \left( \Phi(t,x) \right)$, we have
\begin{align*}
\partial_{y_k} u(y,t) = \partial_{y_k} \widetilde{u}(\Psi_t(y),t)  = \sum_{i=1}^n \partial_{x_i} \widetilde{u}(x,t) \partial_{y_k} (\Psi_t)_{i} \left( \Phi(t,x) \right),
\end{align*}
which leads to
\begin{align*}
| \nabla u(y,t) |^2  = \sum_{i,j,k=1}^n \partial_{y_k} (\Psi_t)_{i} (\Phi(t,x)) \partial_{x_i} \widetilde{u}(x,t) \partial_{y_k} (\Psi_t)_{j} (\Phi(t,x))  \partial_{x_j} \widetilde{u}(x,t).
\end{align*}
Thus, we can define
\begin{align*}
A_{i,j}(t,x):= \sum_{k=1}^n \left[ \partial_{y_k} (\Psi_t)_{i} (\Phi(t,x)) \partial_{y_k} (\Psi_t)_{j} (\Phi(t,x))  \right] J(t,x)
\end{align*}
and deduce
\begin{align*}
\int_{\Omega_t^\text{ext}} | \nabla u(y,t) |^2 \,\mathrm{d}y = \int_{\Omega^\text{ext}} \sum_{i,j=1}^n \partial_{x_i} \widetilde{u} (x,t) \partial_{x_j} \widetilde{u} (x,t) A_{i,j}(t,x) \, \mathrm{d}x.
\end{align*} 
This results in
\begin{align}\label{eq:lambdadarst}
\lambda_1^\alpha(t) = \frac{\int_{\Omega^\text{ext}} \sum_{i,j=1}^n \partial_{x_i} \widetilde{u} (x,t) \partial_{x_j} \widetilde{u} (x,t) A_{i,j}(t,x) \, \mathrm{d}x + \alpha \int_{\partial \Omega} \widetilde{u}(x,t)^2 m(t,x)\,  \mathrm{d}S}{\int_{\Omega^\text{ext}} \widetilde{u}(x,t)^2 J(t,x) \, \mathrm{d}x}
\end{align}
or alternatively, if we define
\begin{align*}
L_{A(t,x)} := \sum_{i,j=1}^n \partial_{x_j} \left( A_{i,j}(t,x) \partial_{x_i} \right) \, \text { and } \, \partial_{\nu_{A(t,x)}} := \sum_{i,j=1}^n \nu_i A_{i,j}(t,x) \partial_{x_j},
\end{align*}
we obtain that \eqref{eq:problemgest} is equivalent to
\begin{align}\label{eq:problemtrafo}
\begin{cases}
L_{A(t,x)} \widetilde{u}(x,t) + \lambda^\alpha(t) \widetilde{u}(x,t) J(t,x) = 0  & \, \text{ in } \Omega^{\text{ext}}, \\
-\partial_{\nu_{A(t,x)}} \widetilde{u}(x,t) + \alpha m(t,x) \widetilde{u}(x,t) = 0 & \,  \text{ on } \partial \Omega.
\end{cases}
\end{align}
In \cite[Theorem 5.5, p. 489]{Chow}, it is demonstrated that when $\lambda_1^\alpha$ is a simple eigenvalue, the associated eigenfunction is analytic in $t$. Moreover, $J, A$ and $m$ are smooth in $t$ since $\partial \Omega$ and $\Phi$ are smooth.  Consequently, we are able to differentiate \eqref{eq:lambdadarst} with respect to $t$.

\subsubsection{First and Second Variation}\label{sec:firandsecvar}

The expressions for the first and second variation of $\lambda_1^\alpha(t)$ can be obtained analogously to those for the Robin eigenvalue on bounded domains, as discussed in \cite[Section 11]{Wagner1}. The proofs of Corollary \ref{koro:firstvarlambdat0} and Theorem \ref{satz:dotlambdaext}  are given in Section \ref{secA2}. Notably, the formula for the first variation is independent of the shape derivative, allowing a straightforward evaluation if the eigenfunction is known.
\vspace{1ex}
\begin{theorem}\label{satz:dotlambdaext}
The first and second variation of $\lambda_1^\alpha$ are given by
\begin{align}\label{eq:lambdadot1}
\dot{\lambda}_1^\alpha(t) =& \int_{\Omega^{\text{ext}}} \left( \nabla \widetilde{u}(x,t) \right)^T \dot{A}(t,x) \nabla \widetilde{u}(x,t) - \lambda_1^\alpha(t)  {\widetilde{u}}(x,t)^2  \dot{J}(t) \, \mathrm{d}x \nonumber \\
&+\alpha \int_{\partial \Omega}  \widetilde{u}(x,t)^2 \dot{m}(t,x)\,  \mathrm{d}S, \\
\ddot{\lambda}_1^\alpha(t) =&\int_{\Omega^{\text{ext}}} \left( \nabla \widetilde{u}(x,t) \right)^T \ddot{A}(t,x) \nabla \widetilde{u}(x,t)    -2  (\nabla \dot{\widetilde{u}}(x,t))^T {A}(t,x) \nabla \dot{\widetilde{u}}(x,t)  \, \mathrm{d}x \nonumber \\
&+\alpha \int_{\partial \Omega}  \widetilde{u}(x,t)^2 \ddot{m}(t,x)   -2    m(t,x) \dot{\widetilde{u}}(x,t)^2 \,  \mathrm{d}S \nonumber \\
&- \dot{\lambda}_1^\alpha(t) \int_{\Omega^{\text{ext}}} {\widetilde{u}}(x,t)^2  \dot{J}(t,x) - \dot{\widetilde{u}}(x,t)  \widetilde{u}(x,t) J(t,x)  \, \mathrm{d}x \nonumber \\
&- \lambda_1^\alpha(t) \int_{\Omega^{\text{ext}}} {\widetilde{u}}(x,t)^2  \ddot{J}(t,x) - 2\dot{\widetilde{u}}(x,t)^2 J(t,x) \, \mathrm{d}x. \nonumber
\end{align}
\end{theorem}
\vspace{1ex}
Evaluating $\dot{\lambda}_1^\alpha(t)$ at $t=0$ allows a significant simplification. The resulting formula depends only on the eigenfunction, $\alpha$, and the mean curvature of $\partial \Omega$. We use the notation $u(x):=\widetilde{u}(x,0)$  and omit arguments of the function where context permits.

\vspace{1ex}
\begin{korollar}\label{koro:firstvarlambdat0}
It holds
\begin{align}\label{eq:firstvarlambdat0}
\dot{\lambda}_1^\alpha(0) =& -\int_{\partial \Omega} \langle v, \nu \rangle \left[ | \nabla u|^2 - 2 \alpha^2 u^2 +\alpha u^2 (n-1) H(\Omega^\text{ext})  - \lambda_1^\alpha(0) u^2 \right]\,  \mathrm{d}S,
\end{align}
where $H(\Omega^\text{ext}) = - H(\Omega)$ denotes the mean curvature of $\partial \Omega^\text{ext}$. 
\end{korollar}
\vspace{1ex}

If we consider $\Omega = B_R \subseteq \mathbb{R}^n$ for $n \geq 2$, the first eigenfunction is given in \eqref{eq:uExtBall}. Since $u$ is radial and $H$ is constant on $\partial B_R$, it holds
\begin{align*}
| \nabla u|^2 - 2 \alpha^2 u^2 +\alpha u^2 (n-1) H(B_R^\text{ext})  - \lambda_1^\alpha(0) u^2 = const \, \text{ on } \, \partial B_R.
\end{align*} 
In view of $-\partial_\nu u(x) + \alpha u=0$ on $\partial B_R$, we obtain $|\nabla u|^2 = (\partial_\nu u)^2 = \alpha^2 u^2$ and it holds $ H(B_R^\text{ext}) = - \frac{1}{R}$.  Hence, by defining $K:=\alpha^2 + \alpha \frac{n-1}{R} + \lambda_1^\alpha(0)$ and writing $u(R)$ as $u(x)$ for $x \in \partial B_R$, \eqref{eq:firstvarlambdat0} simplifies to
\begin{align}\label{eq:dotlambda1radial}
\dot{\lambda}_1^\alpha(0) =  u^2(R) \left( \alpha^2 + \alpha \frac{n-1}{R} + \lambda_1^\alpha(0) \right) \int_{\partial \Omega} \langle v, \nu \rangle \,  \mathrm{d}S =  u^2(R) K \int_{\partial \Omega} \langle v, \nu \rangle \,  \mathrm{d}S.
\end{align}
Given $\dot{V}(0) = \int_{\partial \Omega} \langle v, \nu \rangle \, \mathrm{d} S$, we obtain $\dot{\lambda}_1^\alpha(0) =  u^2(R) K \dot{V}(0)$. Thus, $\dot{\lambda}_1^\alpha(0)$ vanishes for all measure preserving perturbations of the ball. To determine whether the ball is a local minimizer, maximizer, or neither, we are interested in the sign of the second variation. In addition to conditions \eqref{enum:1} - \eqref{enum:perturexta}, we now also assume
\vspace{1ex}
\begin{enumerate}[(a)]
  \setcounter{enumi}{4}
\item $\Phi$ is measure preserving of second order, i.e. $\dot{V}(0) = \ddot{V}(0) = 0$,
\item $\Phi$ satisfies the barycenter condition, i.e. $\int_{\partial \Omega} \langle v, \nu \rangle x \, \mathrm{d}S = 0$.\label{enum:f}
\end{enumerate}
\vspace{1ex}
A useful property of the shape derivative  arises from differentiating \eqref{eq:problemtrafo} with respect to $t$, as elaborated in \cite[(6.2.12), Lemma 6.1]{Wagner1}. This leads to the following lemma.
\vspace{1ex}
\begin{lemma}\label{lemma:u'Kugel}
Let $u$ denote the first eigenfunction of \eqref{eq:problemgest} for $\Omega = B_R$. Then, $u'$ satisfies
\begin{align*}
\begin{cases}
\Delta u' + \lambda_1^\alpha(0) u' = 0   &\text{ in } B_R^{ext} ,\\
-\partial_\nu u' + \alpha u' = -K u \langle v, \nu \rangle &\text{ on } \partial B_R.
\end{cases}
\end{align*}
\end{lemma}
We can use Lemma \ref{lemma:u'Kugel} as in \cite[Corollary 3.6]{BundPhD}, to simplify the second variation, detailed in Section \ref{secA2}. The result is
\begin{align}\label{eq:ddotlambda0}
\ddot{\lambda}_1^\alpha(0) =2 u^2(R) \alpha K \int_{ \partial B_R}    \langle v, \nu \rangle^2  \,  \mathrm{d}S +\alpha u^2(R) \ddot{\mathcal{S}}(0)  -2\mathcal{Q}(u'),
\end{align}
where $\mathcal{S}(t) := | \partial \Omega_t|$ and
\begin{align}\label{eq:DefQ}
\mathcal{Q}(u') := \int_{B_R^{\text{ext}}} | \nabla u' |^2 - \lambda_1^\alpha(0)  \left( u' \right)^2 \, \mathrm{d}x + \alpha \int_{\partial B_R} \left( u' \right) ^2\,  \mathrm{d}S.
\end{align}
The main challenge at this stage is determining the sign of $\ddot{\lambda}_1^\alpha(0)$.

\subsection{The Sign of the Second Variation for the Ball}\label{signofsec}

Due to the technical complexity of calculating $\ddot{\lambda}_1^\alpha(0)$, we introduce the following notation, where $y = f_n(z)$ and $R^2 K = a_n(z)$ as per \eqref{eq:uExtBall}.

\vspace{1ex}
\begin{definition}\label{def:fn}
For $\alpha < \alpha^*$ and $\lambda_1^\alpha < 0$, we set $y:= - \alpha R$ and $z:= \sqrt{- \lambda_1^\alpha(0)} R$. Furthermore, we define the functions
\begin{align*}
&f_n: (0, \infty) \to \left( n-2, \infty \right),   &&x \mapsto x \frac{K_{\frac{n}{2}}(x)}{K_{\frac{n-2}{2}}(x)},\\
&a_n: (0, \infty) \to \R,  &&x \mapsto  f_n(x)^2 - (n-1) f_n(x) - x^2.
\end{align*}
\end{definition}

\subsubsection{Inequalities Involving the Modified Bessel Functions}\label{sec:ineqbess}

We start by establishing some useful inequalities involving $f_n$ and $a_n$. While \eqref{eq:segura} provides a lower bound for the ratio of Bessel functions, the subsequent lemma introduces an alternative lower bound. Specifically, for $\frac{(n-2)^2-1}{4} < z$, Lemma \ref{lemma:fmonfal} enhances \eqref{eq:segura}, and as $z$ tends to infinity, it becomes sharp. 
\vspace{1ex}
\begin{lemma}\label{lemma:fmonfal}
For all $z \in (0,\infty)$ and $n \geq 3$, it holds that
\begin{align*}
z \frac{K_{\frac{n}{2}}(z)}{K_{\frac{n-2}{2}}(z)} = f_n(z)  \geq z + \frac{n-1}{2}.
\end{align*}
\end{lemma}
\begin{proof}
We prove the statement by induction. To this end, we consider $n=3, \ldots,6$ separately. The Bessel functions of half integer order can be expressed by elementary functions (see e.g. \cite[Chapter 9]{abramowitz1968handbook}), yielding
\begin{align*}
K_{\frac{1}{2}}(z) &= \sqrt{\frac{\pi}{2}} \frac{e^{-z}}{\sqrt{z}}, \, \, K_{\frac{3}{2}}(z) =  \sqrt{\frac{\pi}{2}} \frac{e^{-z}}{\sqrt{z}} \frac{z+1}{z}, \, \, K_{\frac{5}{2}}(z) = \sqrt{\frac{\pi}{2}} \frac{e^{-z}}{\sqrt{z}} \frac{z^2+3z+3}{z^2}.
\end{align*}
Hence, Lemma \ref{lemma:fmonfal} directly follows for $n \in \{3,5 \}$. From \cite[Remark 3.11]{Bessel2}, we obtain $\frac{2(z+1)^2}{z(2z+1)} < \frac{K_2(z)}{K_1(z)} < \frac{4z^2+9z+6}{z(4z+3)} $. Therefore,
\begin{align*}
f_4(z) &= z \frac{K_2(z)}{K_1(z)} > \frac{2(z+1)^2}{2z+1} = z + \frac{3}{2} + \frac{1}{2(2z+1)} >  z + \frac{3}{2}.
\end{align*}
By using $\frac{K_{m+1}(x)}{K_m(x)} = \frac{K_{m-1}(x)}{K_m(x)} + \frac{2m}{x}$,
we obtain
\begin{align*}
f_6(z) &= 4 + z \frac{K_1(z)}{K_2(z)} > 4+ z \frac{z(4z+3)}{4z^2+9z+6} = z + \frac{5}{2} + \frac{15z+18}{8z^2+18z+12} >z + \frac{5}{2}.
\end{align*}
Thus, Lemma \ref{lemma:fmonfal} is shown for $3 \leq n \leq 6$. For $n \geq 7$, the recurrence relation \eqref{eq:RecK} yields
\begin{align}\label{eq:nochmal}
f_n(z) = z \frac{K_{\frac{n}{2}}(z)}{K_{\frac{n-2}{2}}(z)} = z  \frac{K_{\frac{n-4}{2}}(z)}{K_{\frac{n-2}{2}}(z)} +  n-2  =  \frac{z^2}{f_{n-2}(z)} + n-2.
\end{align}
Assuming $f_n(z) \geq z + \frac{n-1}{2}$, we obtain
\begin{align*}
& \,f_{n+4}(z) - z - \frac{n+4-1}{2} = \frac{z^2}{f_{n+2}(z)} +n +2- z - \frac{n+3}{2} \\
=&\,\frac{z^2}{\frac{z^2}{f_n(z)} + n} +n+2 - z - \frac{n+3}{2}   \\
=& \, \frac{1}{z^2 + n  f_n(z)} \left( f_n(z) \left[ z^2  -n z + n\frac{n+1}{2}\right]  -z^3  + z^2 \frac{n+1}{2}   \right) \\
>& \, \frac{1}{z^2 + n  f_n(z)} \left( \left[  z + \frac{n-1}{2} \right] \left[ z^2  -n z + n\frac{n+1}{2}\right]  -z^3  + z^2 \frac{n+1}{2}   \right) \\
=& \, \frac{1}{z^2 + n  f_n(z)} n \left[ z + \frac{n-1}{2}\frac{n+1}{2}   \right] > 0.
\end{align*}
Thus, the statement is shown by induction.
\end{proof}

 Lemma \ref{lemma:fmonfal} does not hold for $n=2$, since, by \eqref{eq:asymptoticsalpha}, $\lim_{z \to 0} f_2(z) = 0$. The following lemma demonstrates the negativity of $K$. 
\vspace{1ex}
\begin{lemma}\label{bem:a_n}
For all $z \in (0, \infty), n \geq 2$ it holds that $a_n(z) < 0$.
\end{lemma}
\begin{proof}
We assume there exists a $z_0 > 0$ with $a_n(z_0)=0$, and denote $f_n(z_0) = y_0$. Then, it holds $y_0^2 - (n-1) y_0 - z_0^2 = 0$, which is equivalent to $z_0 = \sqrt{y_0^2-(n-1)y_0}$. This implies
\begin{align*}
f_n \left( \sqrt{y_0^2-(n-1)y_0} \right) = y_0.
\end{align*}
Using \eqref{eq:segura}, we obtain
\begin{align*}
f_n \left( \sqrt{y_0^2-(n-1)y_0}\right) - y_0 &< \frac{n-1}{2} + \sqrt{\frac{(n-1)^2}{4}+y_0^2-(n-1)y_0}-y_0 \\
&= \frac{n-1}{2} + \sqrt{\left[y_0-\frac{n-1}{2} \right]^2}-y_0 = 0.
\end{align*}
Here, we used \eqref{eq:asymptoticsalpha}, i.e. $y_0  \geq \lim_{z \to 0} f_n(z) = n-2$. Hence, $a_n(z)$ has no zeros in $(0, \infty)$. 

For $n \geq 3$, we have $\lim_{z \to 0} a_n(z) = -(n-2) < 0$, and $a_n(z)$ is continuous. Thus, $a_n(z) < 0$ for all $n \geq 3$ and $z \in (0,\infty)$.

For $n=2$, the asymptotic behavior of the Bessel functions for small arguments, as given in \eqref{eq:BesselKasymp0}, yields
\begin{align*}
a_2(z) &= \left( z \frac{K_1(z)}{K_0(z)} \right)^2 - z \frac{K_1(z)}{K_0(z)} - z^2  \approx  \frac{1}{-\ln(\frac{z}{2})-\gamma}\left(  \frac{1}{-\ln(\frac{z}{2})-\gamma} -1\right) - z^2,
\end{align*}
where $\gamma$ is Euler's constant. Hence, $a_2(z)$ is also negative for small $z$, and we conclude $a_2(z) < 0$ for all $z > 0$.
\end{proof}
While Lemma \ref{bem:a_n} provides an upper bound for $a_n$, we also establish a lower bound in the next lemma.
\vspace{1ex}
\begin{lemma}\label{lemma:an}
For all $z>0$, $n \geq 2$, it holds that
\begin{align}\label{eq:lemma:an}
a_n(z) \geq \begin{cases}
-(n-2)  & \, \text{ for } n \geq 3, \\
-\frac{1}{2} & \,  \text{ for } n=2.
\end{cases}
\end{align}
\end{lemma}
\begin{proof}
To prove the claim by induction, we establish the statement separately for $n = 3, \ldots, 6$. For odd $n$, we have, analogously to the proof of Lemma \ref{lemma:fmonfal} ,
\begin{align*}
a_3 (z) &= \left( z \frac{K_{\frac{3}{2}}(z)}{K_{\frac{1}{2}}(z)} \right)^2 - 2z \frac{K_{\frac{3}{2}}(z)}{K_{\frac{1}{2}}(z)}-z^2 = \left( z \frac{z+1}{z} \right)^2 - 2z \frac{z+1}{z}  -z^2 = -1,\\
a_5(z) &= \left( z \frac{K_{\frac{5}{2}}(z)}{K_{\frac{3}{2}}(z)} \right)^2 - 4z \frac{K_{\frac{5}{3}}(z)}{K_{\frac{1}{2}}(z)}-z^2 =\frac{-2z^2-6z-3}{(z+1)^2} > -3.
\end{align*}
For even $n$, we utilize \eqref{eq:nochmal} to obtain
\begin{align*}
a_4(z) = \left( \frac{z^2}{f_{2}(z)} + 2 \right)^2-3\left( \frac{z^2}{f_{2}(z)} + 2 \right)-z^2 = -2 + z^2 \left( \frac{z^2}{f_2^2(z)  }+\frac{1}{f_2(z)} -1 \right).
\end{align*}
Now, \eqref{eq:segura} with $m=0$ yields $ f_2(z) < \frac{1}{2} + \sqrt{\frac{1}{4}+z^2}$. Therefore, 
\begin{align*}
\frac{z^2}{f_2^2(z)  }+\frac{1}{f_2(z)} -1 >&\frac{z^2}{\left( \frac{1}{2} + \sqrt{\frac{1}{4}+z^2} \right)^2}+\frac{1}{\frac{1}{2} + \sqrt{\frac{1}{4}+z^2}} -1 = 0.
\end{align*}
Thus, we can conclude $a_4(z) > -2$. Proceeding in the same manner yields
\begin{align*}
a_6(z) &= \left( \frac{z^2}{f_4(z)} +4 \right)^2 -5\left( \frac{z^2}{f_4(z)} +4 \right)-z^2 =-4 + z^2 \left( \frac{z^2}{f_4^2(z)} + \frac{3}{f_4(z)} -1 \right).
\end{align*}
As before, \eqref{eq:segura} yields $f_4(z) < \frac{3}{2} + \sqrt{\frac{9}{4}+z^2}$ which implies $a_6(z) >-4$.

The same procedure applies for any $n \geq 3$, i.e.
\begin{align*}
a_{n+4}(z) =& \left( n+2 + \frac{z^2 f_n(z)}{z^2 + n f_n(z)} \right)^2 - (n+4-1) \left( n+2 + \frac{z^2 f_n(z)}{z^2 + n f_n(z)} \right) -z^2\\
=&-(n+2) + \frac{z^2}{(n f_n(z)+z^2)^2} \left( f_n^2(z) ( n +z^2) - (n-1)z^2 f_n(z) -z^4  \right)  \\
=& -(n+2) + \frac{z^2}{(n f_n(z)+z^2)^2} \left( z^2 a_n(z) + n f_n^2(z)  \right). 
\end{align*}
Assuming $a_n(z) > -(n-2)$, the last bracket is positive because of
\begin{align*}
 z^2 a_n(z) + n f_n^2(z) > -z^2 (n-2) + n z^2 = 2z^2 > 0,
\end{align*}
where we used $f_n(z) > z$ (see \eqref{eq:segura}). Thus, \eqref{eq:lemma:an} follows inductively for all $n \geq 3$.

We now consider the case $n=2$. Assuming the existence of a positive $z$ such that $f_2^2(z) - f_2(z) -z^2 + \frac{1}{2}=0$, we would have
\begin{align*}
f_2(z) = \frac{1}{2} \pm \sqrt{z^2 - \frac{1}{4}}.
\end{align*}
This expression is only well-defined for $z \geq \frac{{1}}{2}$. By \cite[Theorem 3.10]{Bessel2} it holds 
\begin{align*}
f_2(x) > x \left( 1+ \frac{1}{2x+\frac{1}{2}} \right) > \frac{1}{2} + \sqrt{x^2 - \frac{1}{4}}.
\end{align*} 
Thus, we conclude $f_2(z)>\frac{1}{2} \pm \sqrt{z^2 - \frac{1}{4}}$. Consequently, $f_2^2(z) - f_2(z) -z^2 + \frac{1}{2}=0$ has no solution. Furthermore, we observe $\lim_{z \to 0} f_2^2(z) - f_2(z) -z^2 + \frac{1}{2}= \frac{1}{2}$, thereby confirming the statement.
\end{proof}

\subsubsection{Series Expansion}
To evaluate the sign of the second variation,  we adopt the approach outlined in \cite[Section 8]{Wagner1}, which entails representing $u'$ and $\langle v, \nu \rangle$ in an appropriate orthonormal basis. To this end, we employ the Steklov eigenvalue problem
\begin{align}\label{eq:SteklovExt}
\begin{cases}
\Delta \phi_k + \lambda_1^\alpha(0) \phi_k = 0  &\text{ in } B_R^{\text{ext}} ,\\
-\partial_\nu \phi_k + \alpha \phi_k = \mu_k \phi_k &\text{ on } \partial B_R.
\end{cases}
\end{align}
For $\lambda_1^\alpha(0) \neq 0$, the corresponding bilinear form
\begin{align*}
a: W^{1,2}(\Omega^{\text{ext}}) \times W^{1,2}(\Omega^{\text{ext}}) \to \R, \, \,  (u,v) \mapsto \int_{\Omega^{\text{ext}}} \langle \nabla u, \nabla v \rangle \, \mathrm{d}x -\lambda_1^\alpha(0) \int_{\Omega^{\text{ext}}}  u v \, \mathrm{d}x
\end{align*}
fulfills 
\begin{align*}
\min \{ 1, | \lambda_1^\alpha(0) | \} ||u ||_{W^{1,2}(\Omega^{\text{ext}})}^2 \leq a(u,u) \leq \max \{ 1, | \lambda_1^\alpha(0) | \} ||u ||_{W^{1,2}(\Omega^{\text{ext}})}^2.
\end{align*}
Hence, \eqref{eq:SteklovExt} possesses a sequence of eigenvalues $ \left( \mu_k \right)_{k \in \mathbb{N}} $ accumulating at infinity. The corresponding eigenfunctions $\left( \phi_k \right)_{k \in \mathbb{N}} \subset W^{1,2}(\Omega^{\text{ext}})$ form an orthogonal basis, as discussed in \cite[Section 4]{AuchmutyNeu}. Consequently, we can write $u' = \sum_{k=0}^\infty c_k \phi_k$. Notably, in contrast to Section \ref{subsc:charal}, we operate within the space $W^{1,2}(\Omega^{\text{ext}})$ instead of $E^1(\Omega^{\text{ext}})$.

To solve \eqref{eq:SteklovExt}, we set $x= r \cdot \xi$ for $r \in [R, \infty), \xi \in \mathbb{S}^{n-1}$ and use the ansatz $\phi(x) = a(r) b(\xi)$. Expressing the Laplace operator as
\begin{align*}
\Delta = \frac{\partial^2}{\partial r^2} + \frac{n-1}{r} \frac{\partial }{\partial r} + \frac{1}{r^2} \Delta_{\mathbb{S}^{n-1}}	,
\end{align*}
where $\Delta_{\mathbb{S}^{n-1}}$ denotes the Laplace-Beltrami operator, we obtain
\begin{align*}
\Delta \phi(x) + \lambda_1^\alpha(0) \phi(x) = b(\xi) \left[ \frac{\mathrm{d}^2 a(r) }{\mathrm{d} r^2} + \frac{n-1}{r}  \frac{\mathrm{d} a(r) }{\mathrm{d} r}  + \lambda_1^\alpha(0) a(r) \right] + \frac{a(r)}{r^2} \Delta_{\mathbb{S}^{n-1}}  b(\xi).
\end{align*}
We set $b(\xi) = Y_{k,i}(\xi)$, representing the $i-$th spherical harmonic of degree $k \in \mathbb{N}_0$. The dimension of the space of spherical harmonics of degree $k$ is given by
\begin{align*}
d_k := \binom{n+k-1}{n-1} - \binom{n+k-3}{n-1},
\end{align*}
and  $\{Y_{k,i} \}_{k \in \N_0, i=1,\ldots,d_k}$ is an orthonormal basis of $L^2(\partial B_1)$. Furthermore, it holds
\begin{align*}
\Delta_{\mathbb{S}^{n-1}}  Y_{k,i}(\xi) +k (n+k-2)Y_{k,i}(\xi) = 0.
\end{align*}
Hence,
\begin{align*}
\Delta \phi(x) + \lambda_1^\alpha(0) \phi(x) = b(\xi) \left[ \frac{\mathrm{d}^2 a(r) }{\mathrm{d} r^2} + \frac{n-1}{r}  \frac{\mathrm{d} a(r) }{\mathrm{d} r} - \frac{a(r)}{r^2} k(n+k-2)   + \lambda_1^\alpha(0) a(r) \right].
\end{align*}
Therefore, $a(r)$ must satisfy
\begin{align*}
\frac{\mathrm{d}^2 a(r)}{\mathrm{d} r^2} + \frac{\mathrm{d} a(r)}{\mathrm{d}r} \frac{n-1}{r} - \frac{k(k+n-2)}{r^2} a(r) + \lambda_1^\alpha(0) a(r) = 0,
\end{align*}
leading to
\begin{align*}
a_{k,i}(r) = p_{k,i}  r^{-\frac{n-2}{2}} K_{\frac{n+2k-2}{2}}\left( \sqrt{-\lambda_1^\alpha(0)}r\right) + q_{k,i}  r^{-\frac{n-2}{2}} I_{\frac{n+2k-2}{2}}\left( \sqrt{-\lambda_1^\alpha(0)}r\right),
\end{align*}
with $p_{k,i},q_{k,i} \in \R$. In view of the asymptotic behavior of the Bessel functions given in \eqref{eq:BesselIasympinf} and \eqref{eq:BesselKasympinf}, we have to choose $q_{k,i} =0$ to ensure that $\phi \in L^2(\Omega^{\text{ext}})$. Hence, it holds
\begin{align*}
\frac{-\frac{\mathrm{d}}{\mathrm{d}r}\phi_{k,i} + \alpha \phi_{k,i}}{\phi_{k,i}}=& \frac{  -\frac{\mathrm{d}}{\mathrm{d}r} \left(  r^{-\frac{n-2}{2}} K_{\frac{n+2k-2}{2}}( \sqrt{-\lambda_1^\alpha(0)} r) \right) }{r^{-\frac{n-2}{2}} K_{\frac{n+2k-2}{2}}( \sqrt{-\lambda_1^\alpha(0)} r)} + \alpha   \\
=& -\frac{k}{r} + \frac{    \sqrt{-\lambda_1^\alpha(0)}  K_{\frac{n+2k}{2}}(\sqrt{-\lambda_1^\alpha(0)}r)       }{  K_{\frac{n+2k-2}{2}}( \sqrt{-\lambda_1^\alpha(0)}r)} + \alpha.
\end{align*}
Thus, the boundary condition of \eqref{eq:SteklovExt} yields
\begin{align}\label{eq:Steklovextalpha}
\mu_{k,i}(z) &= \alpha + \frac{z}{R} \frac{K_{\frac{n+2k}{2}}(z)}{K_{\frac{n+2k-2}{2}}(z)} - \frac{k}{R} = -\frac{f_n(z)}{R} + \frac{f_{n+2k}(z)}{R} - \frac{k}{R}.
\end{align}
As this is independent of $i$, we write $\mu_k := \mu_{k,i}$. Additionally, the eigenvalues $\mu_k$ exhibit the same ordering as the spherical harmonics as demonstrated in the next lemma. In particular, Lemma \ref{lemma:Order} shows that $\mu_k(z)>0$ for $k \geq 1$ and $z \in (0, \infty)$.

%
\vspace{1ex}
\begin{lemma}\label{lemma:Order}
Let $n \geq 2$. The Steklov eigenvalues are ordered analogously to the spherical harmonics, i.e. $\mu_k(z) < \mu_{k+1}(z)$ holds for all $k \in \mathbb{N}_0$ and $z \in (0, \infty)$. 
\end{lemma}
\begin{proof}
For $k \in \N$, the inequality $\mu_{k+1}(z) > \mu_{k}(z)$ is equivalent to
\begin{align*}
 \, f_{n+2k+2}(z)-1 >& \, f_{n+2k}(z) \\
\Leftrightarrow \, \frac{z^2}{f_{n+2k}(z)}+n+2k-1 >& \, f_{n+2k}(z) \\
\Leftrightarrow \, z^2 + (n+2k-1)f_{n+2k}(z)-f_{n+2k}(z)^2 >& \, 0\\
 \Leftrightarrow \, -a_{n+2k}(z) >& \, 0.
\end{align*}
In  Lemma \ref{bem:a_n} it  has been established that $a_m(z) < 0$ for all $z \in (0, \infty), m\geq 2$.  Since $\mu_0 = 0$, it remains to show $\mu_1 > 0$. But analogously we obtain $\mu_1 > 0 \Leftrightarrow -a_n > 0$. Hence, the assertion is proven.
\end{proof}

As mentioned earlier, we express the shape derivative in terms of eigenfunctions of \eqref{eq:SteklovExt} to determine the sign of the second variation. If $u' = \sum_{k=0}^\infty \sum_{i=1}^{d_k}  c_{k,i} \phi_{k,i}$, and we define $b_{k,i} :=\frac{c_{k,i} \mu_k}{-K u(R)}$, the boundary condition of Lemma \ref{lemma:u'Kugel} implies
\begin{align*}
\langle v, \nu \rangle = \sum_{k=0}^\infty \sum_{i=1}^{d_k} \frac{  c_{k,i} \left( \partial_\nu \phi_{k,i} - \alpha \phi_{k,i} \right)}{K u(R)} = \sum_{k=0}^\infty \sum_{i=1}^{d_k} \frac{c_{k,i} \mu_k}{-K u(R)} \phi_{k,i}= \sum_{k=0}^\infty \sum_{i=1}^{d_k} b_{k,i} \phi_{k,i}.
\end{align*}
An illustration of such perturbations can be found in \cite[Section 3.3.1]{BundPhD}. Since the spherical harmonics of degree $0$ are constant functions, we obtain
\begin{align*}
\phi_{0,j}  \dot{V}(0)=&  \int_{\partial B_R}  \langle v, \nu \rangle \phi_{0,j}\,  \mathrm{d}S = \int_{\partial B_R}  \sum_{k=2}^\infty \sum_{i=1}^{d_k}  b_{k,i} \phi_{k,i} \phi_{0,j}\,  \mathrm{d}S \\
=& \sum_{k=2}^\infty \sum_{i=1}^{d_k}  b_{k,i} \int_{\partial B_R}   \phi_{k,i} \phi_{0,j} \,  \mathrm{d}S = \sum_{k=2}^\infty \sum_{i=1}^{d_k}  b_{k,i} \delta_{k,0} \delta_{j,i} = b_{0,j} .
\end{align*}
Thus, $\Phi$ is measure preserving if and only if $b_{0,i} = 0$ (and therefore also $c_{0,i}=0$) for all $1 \leq i \leq d_0$. Similarly, by incorporating spherical harmonics of degree $1$, we have
\begin{align*}
\Phi \text{ satisfies the barycenter condition} \Leftrightarrow b_{1,i} = 0 \, \text{  for all } \, 1 \leq i \leq d_1.
\end{align*}

In summary, assuming that the perturbation is measure preserving and satisfies the barycenter condition, we obtain the representations
\begin{align}\label{eq:u'reihe}
u' = \sum_{k=2}^\infty \sum_{i=1}^{d_k}  c_{k,i} \phi_{k,i} \, \, \text{ and }\, \, \langle v, \nu \rangle = \sum_{k=2}^\infty \sum_{i=1}^{d_k}  b_{k,i} \phi_{k,i}
\end{align}
with $b_{k,i} = \frac{c_{k,i} \mu_k}{K u(R)}$. We insert this representation into \eqref{eq:ddotlambda0}. Integration by parts yields
\begin{align*}
\mathcal{Q}(u') &= \int_{B_R^{\text{ext}}} | \nabla u' |^2 - \lambda_1^\alpha(0)  \left( u' \right)^2 \,  \mathrm{d}x + \alpha \int_{\partial B_R} \left( u' \right)^2\,  \mathrm{d}S \\
&= \int_{\partial B_R}  \left( -\partial_\nu u' + \alpha u' \right) u' \,  \mathrm{d}S \\
 &=\int_{\partial B_R} \left(  \sum_{k=2}^\infty \sum_{i=1}^{d_k}  c_{k,i} \left( -\partial_{\nu} \phi_{k,i} + \alpha \phi_{k,i} \right) \right) \left( \sum_{p=2}^\infty \sum_{q=1}^{d_p}  c_{p,q} \phi_{p,q} \right)\,  \mathrm{d}S \\
&= \sum_{k=2}^\infty \sum_{i=1}^{d_k} \sum_{p=2}^\infty \sum_{q=1}^{d_p}  c_{p,q}  c_{k,i}  \mu_{k}  \int_{\partial B_R}    \phi_{k,i}    \phi_{p,q} \,  \mathrm{d}S = \sum_{k=2}^\infty \sum_{i=1}^{d_k}   c_{k,i}^2  \mu_{k}.
\end{align*}
Additionally, it holds 
\begin{align*}
\int_{\partial B_R} \langle v, \nu \rangle^2\,  \mathrm{d}S = \sum_{k=2}^\infty \sum_{i=1}^{d_k}    \sum_{p=2}^\infty \sum_{q=1}^{d_p} b_{k,i}  b_{p,q} \int_{\partial B_R} \phi_{k,i}  \phi_{p,q}\,  \mathrm{d}S =\sum_{k=2}^\infty \sum_{i=1}^{d_k} b_{k,i}^2.
\end{align*}
The second variation of the surface, $\ddot{\mathcal{S}}(0)$ is given in \cite[(2.3.28)]{Wagner1} and \cite[Lemma 13.1]{Wagner1}. Assuming the perturbation is measure preserving of second order, it holds
\begin{align*}
\ddot{\mathcal{S}}(0) &= \frac{1}{R^2} \sum_{k=2}^\infty \sum_{i=1}^{d_k} {k(n+k-2)}  b_{k,i}^2 -\frac{n-1}{R^2} \sum_{k=2}^\infty \sum_{i=1}^{d_k}   b_{k,i}^2.
\end{align*}
Inserting this into \eqref{eq:ddotlambda0} results in
\begin{align}\label{eq:lambda..Kreis}
\ddot{\lambda}_1^\alpha(0) =& \, 2 \alpha u(R)^2 K \int_{\partial B_R} \langle v, \nu \rangle^2\,  \mathrm{d}S + \alpha u^2(R) \ddot{\mathcal{S}}(0) -2 \mathcal{Q}(u') \nonumber \\
=& \,  \alpha u(R)^2 \sum_{k=2}^\infty \sum_{i=1}^{d_k} \left[ 2 K + \frac{1}{R^2} \left[ k(n+k-2) -(n-1) \right] \right] b_{k,i}^2 -2 \sum_{k=2}^\infty \sum_{i=1}^{d_k}   c_{k,i}^2  \mu_{k} \nonumber \\
=& \,   \sum_{k=2}^\infty \sum_{i=1}^{d_k} \alpha u(R)^2 \left[ 2 K + \frac{1}{R^2} \left[ k(n+k-2) -(n-1) \right] \right] b_{k,i}^2 -2 \frac{(b_{k,i} K u(R))^2}{\mu_{k}}  \nonumber \\
=& \,  u(R)^2 \sum_{k=2}^\infty \sum_{i=1}^{d_k} \left[ 2 K  \alpha + \frac{ \alpha}{R^2} \left[ k(n+k-2) -(n-1) \right]  -2 \frac{K^2}{\mu_{k} }\right] b_{k,i}^2.
\end{align}
Supposing $L(\alpha, n,k,R) :=\left[ 2 K  \alpha + \frac{ \alpha}{R^2} \left[ k(n+k-2) -(n-1) \right]  -2 \frac{K^2}{\mu_{k} }\right] < 0$ holds for all $\alpha < \alpha^*, n\geq 2, k \geq 2$ and $R>0$, we can conclude that the second variation is negative. To show $L(\alpha, n,k,R) < 0$, we use  the inequalities from Section \ref{sec:ineqbess}.
\vspace{1ex}
\begin{theorem}\label{satz:lambda..ext}
For all $n \geq 2$, $R>0$ and  $\alpha< \alpha^*(B_R^\text{ext})$ it holds that $\ddot{\lambda}_1^\alpha(0)  < 0$.
\end{theorem}
\begin{proof}
Considering Definition \ref{def:fn} and $R^2 K= a_n$, we obtain
\begin{align*}
 &R^4 \mu_k(z) L(\alpha, n,k,R) \\
=&-2 y a_n(z) R \mu_k(z)  -2 a_n(z)^2 -y R \mu_k(z)  [k^2+(n-2)k-(n-1)]     \\
=&-2 a_n(z)^2 - y  R \mu_k(z) \left[ 2 a_n(z) +k^2+(n-2)k-(n-1) \right].
\end{align*} 
For $n=2$, it holds that
\begin{align*}
R^4 \mu_k(z) L(\alpha, n,k,R) = -2 a_2(z)^2 - y  R \mu_k(z) \left[ 2 a_2(z) +k^2-1 \right].
\end{align*}
Using  Lemma \ref{lemma:an}, i.e. $a_2(z) > - \frac{1}{2}$, we obtain $2 a_2(z) +k^2-1 > k^2 -2 >0$. Hence, $L(\alpha, 2,k,R) < 0$ holds for all $k \geq 2$, $R>0$ and  $\alpha< \alpha^*(B_R^\text{ext})$.

For $n \geq 3$, we distinguish two cases. \\
\textbf{Case 1: $k \geq 3$}. Due to Lemma \ref{lemma:an}, i.e. $a_n(z) > - (n-2)$, it holds
\begin{align*}
&2 a_n(z) +k^2+(n-2)k-(n-1) >  (k-3)n+k^2-2k+5 > 0.
\end{align*}
Thus, we conclude
\begin{align*}
R^4 \mu_k(z) L(\alpha, n,k,R) = -2 a_n(z)^2 - y  R \mu_k(z) \left[ 2 a_n(z) +k^2+(n-2)k-(n-1) \right]  < 0.
\end{align*}
\textbf{Case 2: $k=2$}. For $n \leq 5$, the same reasoning as in Case 1 applies. However, for $n \geq 6$, the calculations become more intricate. It holds 
\begin{align*}
L(\alpha, n,2,R) < 0 \, \Leftrightarrow& \, a_n(z)^2 + R \mu_2(z) y a_n(z) +R \mu_2(z) y \frac{n+1}{2} >0 \\
\Leftrightarrow& \, \left( a_n(z) + \frac{R \mu_2(z) y}{2} \right)^2 - \frac{(R \mu_2(z) y)^2}{4} +R \mu_2(z) y \frac{n+1}{2} >0 \\
\Leftrightarrow& \, \left( a_n(z) + \frac{R \mu_2(z) y}{2} \right)^2 + \frac{R \mu_2(z) y}{2} \left( n+1-\frac{R \mu_2(z) y}{2} \right) >0.
\end{align*}
We aim to prove $\frac{R \mu_2(z) y}{2} < n+1$, which makes the last inequality clearly valid, thus proving $L(\alpha, n,2,R) < 0$. Because of 
\begin{align*}
{R \mu_2(z) y} =  \frac{z K_{\frac{n}{2}}(z)}{K_{\frac{n-2}{2}}(z)} \left( - \frac{K_{\frac{n}{2}}(z) z}{ K_{\frac{n-2}{2}}(z)} + \frac{K_{\frac{n+4}{2}}(z) z}{ K_{\frac{n+2}{2}}(z)} - 2 \right) =   y \left( - y + \frac{K_{\frac{n+4}{2}}(z) z}{ K_{\frac{n+2}{2}}(z)} - 2 \right),
\end{align*}
it holds
\begin{align*}
n+1 - \frac{R \mu_2(z) y}{2} = n+1 - \frac{y}{2} \left( -y+z \frac{K_{\frac{n+4}{2}}(z)}{K_{\frac{n+2}{2}}(z)} -2 \right).
\end{align*}
Applying \eqref{eq:RecK}, we obtain $z \frac{K_{\frac{n+4}{2}}(z)}{K_{\frac{n+2}{2}}(z)} = n+2 + z \frac{K_{\frac{n}{2}}(z)}{K_{\frac{n+2}{2}}(z)}$ and $z \frac{K_{\frac{n}{2}}(z)}{K_{\frac{n+2}{2}}(z)} = \frac{z^2y}{ny+z^2}$. Thus, 
\begin{align*}
n+1 >\frac{R \mu_2(z) y}{2} \, \Leftrightarrow \, n+1 > \frac{y}{2} \left( -y+n+\frac{z^2y}{ny+z^2} \right).
\end{align*}
Given $\lim_{z \to 0} y=\lim_{z \to 0} f_n(z) = n-2$, and considering \eqref{eq:BesselKasympinf}, we have
\begin{align*}
\lim_{z \to 0} n+1 - \frac{y}{2} \left( -y+n+\frac{z^2y}{ny+z^2} \right) > 0, \, \, \, \lim_{z \to \infty} n+1 - \frac{y}{2} \left( -y+n+\frac{z^2y}{ny+z^2} \right) > 0.
\end{align*}
Furthermore, the function $\frac{\mu_2(z) y}{2}$ is continuous.  Thus, if  $n+1 -  \frac{y}{2} \left( -y+n+\frac{z^2y}{ny+z^2} \right)$ has no zeros, we obtain $n+1 > \frac{R \mu_2(z) y}{2}$. Since $z>0$, the only possible zero occurs at
\begin{align*}
z_0= \frac{\sqrt{-(ny-2n-2)ny(ny-y^2-2n-2))}}{ny-2n-2}.
\end{align*}
However, according to \eqref{eq:segura},
\begin{align*}
& \,f_n(z_0)-y > \frac{n-2}{2} + \sqrt{\frac{(n-2)^2}{4}+z_0^2}-y \\
=& \,\frac{1}{2} \left( n-2-2y + \sqrt{\frac{-8+n^3(y-2)+n^2(-4y^2+4y+6)+n(4y^3+12y)}{-2+(y-2)n}} \right).
\end{align*}
This expression can only be negative if $n-2-2y<0$. Consequently,
\begin{align*}
&\frac{1}{2} \left( n-2-2y + \sqrt{\frac{-8+n^3(y-2)+n^2(-4y^2+4y+6)+n(4y^3+12y)}{-2+(y-2)n}} \right) > 0 \\
\Leftrightarrow \, &\sqrt{\frac{-8+n^3(y-2)+n^2(-4y^2+4y+6)+n(4y^3+12y)}{-2+(y-2)n}} > 2y-n+2 \\
\Leftrightarrow \, &\frac{8y(2n+y+2)}{-2+(y-2)n} > 0.
\end{align*}
The numerator is obviously positive, and so is the denominator by the assumption $y > n-2$ and $n \geq 3$. Thus, $n+1 -  \frac{y}{2} \left( -y+n+\frac{z^2y}{ny+z^2} \right)$ has no zeros. Consequently, $L(\alpha,n,2,R)<0$ is shown, concluding the proof.
\end{proof}
Combining Theorem \ref{satz:lambda..ext} and \eqref{eq:dotlambda1radial} yields that the exterior of the ball maximizes $\lambda_1^\alpha(\Omega^\text{ext})$ locally among all nearly spherical domains with prescribed measure. However, the neighborhood of $B_R$ in which the ball maximizes $\lambda_1^\alpha(\Omega^\text{ext})$ might depend on the particular perturbation. We conjecture that an appropriate choice of condition \eqref{enum:ot1} (page \pageref{enum:ot1}) addresses this issue.

 \vspace{1ex}
\begin{remark}
For $k=1$, the perturbation is volume preserving. Consequently, the first variation remains zero. Regarding the second variation, we obtain 
\begin{align*}
R^4 \mu_1(z)L(\alpha,n,1,R) = -2 a_n(z) \left( a_n(z) + y R \mu_1(z) \right).
\end{align*}
Due to
\begin{align*}
y R \mu_1(z) &= f_n(z) \left( -f_n(z) + f_{n+2}(z) -1 \right) = f_n(z) \left( -f_n(z) + \left[ \frac{z^2}{f_n(z)} +n \right] -1 \right) \\ 
&= -f_n(z)^2 + (n-1) f_n(z) +z^2 = -a_n(z),
\end{align*}
we obtain $L(\alpha,n,1,R)=0$. This is in line with the idea that spherical harmonics of degree $1$ merely shift the domain, implying that the first eigenvalue remains unchanged.
\end{remark}

\subsection{Quantitative Inequalities}\label{sec:quantext}
To derive a quantitative inequality for $\lambda_1^\alpha(t)$, we want to find an inequality of the form 
\begin{align}\label{eq:quantidee}
\lambda_1^\alpha(0)-\lambda_1^\alpha(t) \geq c \left[ \mathcal{S}(t)-\mathcal{S}(0) \right] \, \text{ for } \, |t| \text{ small},
\end{align}
where $\mathcal{S}(t) := \mathcal{S}(\Omega_t) = | \partial \Omega_t|$ and $c$ is independent of the perturbation. It is known that the unique minimizer of the surface area, among all domains with prescribed volume, except translations,  is the the ball. Consequently, we have $\mathcal{S}(t) -\mathcal{S}(0) > 0$ for $t \neq 0$. Furthermore, inequalities for $\mathcal{S}(t)-\mathcal{S}(0)$ are known which involve the geometry of $\Omega_t$, see e.g. \cite{fusco2008sharp}. To achieve \eqref{eq:quantidee}, we follow a similar approach as outlined in \cite[Section 10]{Wagner1}. In order to determine $c$, we need Lemma \ref{lemma:quant}. The proof has the same structure as the proof of Theorem \ref{satz:lambda..ext}. A detailed proof of Lemma \ref{lemma:quant} can be found in \cite[Lemma 3.17]{BundPhD}.

For the remainder of this section, we assume $n \geq 2$, and we assume that $\Phi$ is perturbation of $B_R^\text{ext}$ satisfying the conditions \eqref{enum:1} - \eqref{enum:f} (page \pageref{enum:ot1}, \pageref{enum:f}). Additionally, we assume that $\alpha < \alpha^*(\Omega_t^\text{ext})$ for all $|t| < t_0$.

\vspace{1ex}
\begin{lemma}\label{lemma:quant}
It holds
\begin{align*}
\frac{\mathcal{S}(0)}{\lambda_1^\alpha(0)} \frac{\ddot{\lambda}_1^\alpha(0)}{\ddot{\mathcal{S}}(0)} \geq \frac{\alpha u(R)^2}{n+1} \frac{\mathcal{S}(0)}{\lambda_1^\alpha(0)} > 0.
\end{align*}
\end{lemma}
\vspace{1ex}

To prove \eqref{eq:quantidee}, we introduce the following notations.
\vspace{1ex}
\begin{definition}
For $|t| < t_0$ we define
\begin{align*}
\delta \mathcal{S}(t):= \frac{\mathcal{S}(t)}{\mathcal{S}(0)}-1, \, \, \delta \lambda_1^\alpha(t):= \frac{\lambda_1^\alpha(t)}{\lambda_1^\alpha(0)} -1, \, \, \delta(t) :=  \frac{\delta \lambda_1^\alpha(t)}{\delta \mathcal{S}(t)}.
\end{align*}
\end{definition}
\vspace{1ex}
Using the first two definitions, $\delta(t)$ equals $\delta(t) =  - \frac{\mathcal{S}(0)}{\lambda_1^\alpha(0)} \frac{\lambda_1^\alpha(0)-\lambda_1^\alpha(t)}{\mathcal{S}(t)-\mathcal{S}(0)}$. Thus, a lower bound for $\delta$ gives the desired inequality, as shown in Theorem \ref{satz:quantineq}.
\vspace{1ex}
\begin{theorem}\label{satz:quantineq}
We define
\begin{align*}
C_1:=& \sup_{| t | < t_0} \frac{\left| \lambda_1^\alpha(t) - \lambda_1^\alpha(0) - t \dot{\lambda}_1^\alpha(0) - \frac{t^2}{2} \ddot{\lambda}_1^\alpha(0) \right|}{t^3}, \\
C_2:=& \sup_{| t | < t_0} \frac{\left| \mathcal{S}(t) - \mathcal{S}(0) - t \dot{\mathcal{S}}(0) - \frac{t^2}{2} \ddot{\mathcal{S}}(0) \right|}{t^3}.
\end{align*}
Then, for every $\varepsilon > 0$ and $|t| < \min \left\{ \frac{-  \varepsilon \ddot{\lambda}_1^\alpha(0)  }{2C_1} , \frac{  \varepsilon \ddot{\mathcal{S}}(0)  }{2C_2}, t_0 \right\}$, it holds
\begin{align*}
\lambda_1^\alpha(0)-\lambda_1^\alpha(t) > -\alpha u(R)^2 \frac{\mathcal{S}(t)-\mathcal{S}(0)}{n+1}   \frac{1- \varepsilon}{1+ \varepsilon}.
\end{align*}
\end{theorem}
\begin{proof}
Considering $\dot{\mathcal{S}}(0) = \dot{\lambda}_1^\alpha(0)=0$, we obtain
\begin{align*}
\delta \lambda_1^\alpha(t) &= \frac{\lambda_1^\alpha(t) - \lambda_1^\alpha(0)}{\lambda_1^\alpha(0)} \geq \frac{  t \dot{\lambda}_1^\alpha(0) + \frac{t^2}{2} \ddot{\lambda}_1^\alpha(0) + t^3 C_1}{\lambda_1^\alpha(0)} = \frac{   \frac{t^2}{2} \ddot{\lambda}_1^\alpha(0) + t^3 C_1}{\lambda_1^\alpha(0)} , \\
\delta \mathcal{S}(t) &=  \frac{\mathcal{S}(t)-\mathcal{S}(0)}{\mathcal{S}(0)} \leq \frac{  t \dot{\mathcal{S}}(0) + \frac{t^2}{2} \ddot{\mathcal{S}}(0) + t^3 C_2}{\lambda_1^\alpha(0)} = \frac{  \frac{t^2}{2} \ddot{\mathcal{S}}(0) + t^3 C_2}{\lambda_1^\alpha(0)} .
\end{align*}
This leads to
\begin{align*}
\delta(t) \geq \frac{\frac{t^2}{2} \frac{\ddot{\lambda}_1^\alpha(0)}{\lambda_1^\alpha(0)} + t^3 \frac{C_1}{\lambda_1^\alpha(0)}}{\frac{t^2}{2} \frac{\ddot{\mathcal{S}}(0)}{\mathcal{S}(0)} + t^3 \frac{C_2}{\mathcal{S}(0)}} = \frac{\mathcal{S}(0)}{-\lambda_1^\alpha(0)} \frac{- \ddot{\lambda}_1^\alpha(0) - 2t C_1}{ \ddot{\mathcal{S}}(0) + 2t C_2}.
\end{align*}
By choosing $|t| < \min \left\{ \frac{-  \varepsilon \ddot{\lambda}_1^\alpha(0)  }{2C_1} , \frac{  \varepsilon \ddot{\mathcal{S}}(0)  }{2C_2} \right\}$, and considering Lemma \ref{lemma:quant}, it holds
\begin{align*}
\delta(t) \geq \frac{\mathcal{S}(0)}{-\lambda_1^\alpha(0)} \frac{- \ddot{\lambda}_1^\alpha(0) +  \varepsilon \ddot{\lambda}_1^\alpha(0)  }{ \ddot{\mathcal{S}}(0) + \varepsilon \ddot{\mathcal{S}}(0)} = \frac{\mathcal{S}(0)\ddot{\lambda}_1^\alpha(0)}{\lambda_1^\alpha(0)\ddot{\mathcal{S}}(0)} \frac{ 1 -  \varepsilon   }{ 1 + \varepsilon } \geq \frac{\alpha u(R)^2}{n+1} \frac{\mathcal{S}(0)}{\lambda_1^\alpha(0)} \frac{ 1 -  \varepsilon   }{ 1 + \varepsilon }.
\end{align*}
Since $\delta(t) = - \frac{\mathcal{S}(0)}{\lambda_1^\alpha(0)} \frac{\lambda_1^\alpha(0)-\lambda_1^\alpha(t)}{\mathcal{S}(t)-\mathcal{S}(0)}$, we obtain
\begin{align*}
\frac{\lambda_1^\alpha(0)-\lambda_1^\alpha(t)}{\mathcal{S}(t)-\mathcal{S}(0)} > \frac{-\alpha u(R)^2}{n+1}   \frac{1- \varepsilon}{1+ \varepsilon}.
\end{align*}
\end{proof}

We cannot assure that the values of $C_1$ and $C_2$, and thus the interval for $t$, where the inequality holds true, can be chosen independently of the specific perturbation. Nonetheless, we expect that the condition \eqref{enum:ot1} can effectively address this problem.

\begin{appendices}

\section{Properties of Bessel Functions}\label{secA1}

The properties of the modified Bessel functions, presented in this section, are well known and can, for example, be found in \cite[Chapter 9]{abramowitz1968handbook}. For $m \in \R_{\geq 0}$, the modified Bessel functions $I_m(x)$ and $K_m(x)$ are defined as the canonical solutions of
\begin{align}\label{eq:PDEBessel}
x^2 \frac{\mathrm{d}^2f(x)}{\mathrm{d} x^2} + x \frac{\mathrm{d} f(x)}{\mathrm{d} x} - (x^2+m^2) f(x) = 0.
\end{align} 
They satisfy the recurrence relations
\begin{align}
K_{m+2}(x) &= \frac{(2m+2)K_{m+1}(x) }{x}+ K_m(x), \label{eq:RecK} \\
 I_{m+2}(x) &= \frac{-(2m+2)I_{m+1}(x) }{x}+ I_m(x).\label{eq:RecI}
\end{align}
Differentiating these functions with respect to $x$ yields
\begin{align}\label{eq:Bessediff}
\frac{\mathrm{d}}{\mathrm{d}x} K_m(x) = -K_{m+1}(x) + \frac{m K_m(x)}{x} \, \text{ and } \,
\frac{\mathrm{d}}{\mathrm{d}x} I_m(x) = I_{m+1}(x) + \frac{m I_m(x)}{x}. 
\end{align} 
The Bessel functions have the following asymptotic behavior. For $x \to 0$, it holds
\begin{align}
I_m(x) &\approx \frac{1}{\Gamma(m+1)} \left( \frac{x}{2} \right)^m, \label{eq:BesselIasymp0} \\
K_m(x) &\approx \begin{cases}
-\left( \ln\left(\frac{x}{2}\right) + \gamma \right)  & \text{ for } m=0, \\
\frac{\Gamma(m)}{2} \left( \frac{2}{x} \right)^m & \,  \text{ for } m>0,
\end{cases}\label{eq:BesselKasymp0}
\end{align}
where $\gamma$ denotes the Euler constant. For $x \to \infty$, we have the asymptotic behavior
\begin{align}
I_m(x) &= \frac{e^x}{\sqrt{2 \pi x}} \left( 1- \frac{4m^2-1}{8x} + \frac{(4m^2-1)(4m^2-9)}{2(8x)^2} + \mathcal{O}\left( \frac{1}{x^3} \right) \right),\label{eq:BesselIasympinf}\\
K_m(x) &= \sqrt{\frac{\pi}{2x}}e^{-x} \left( 1+ \frac{4m^2-1}{8x} + \frac{(4m^2-1)(4m^2-9)}{2(8x)^2} + \mathcal{O}\left( \frac{1}{x^3} \right) \right).\label{eq:BesselKasympinf}
\end{align}
A relation between $K_m$ and $I_m$ is given by
\begin{align}\label{eq:abel2}
-K_m(x) I_{m+1}(x) = -\frac{1}{x} + I_m(x) K_{m+1}(x).
\end{align}
A useful inequality for the ratio of Bessel functions is proven in \cite[Theorem 1, Theorem 5]{Segura} by J. Segura: For $x > 0, m \geq 0$ it holds that
\begin{align}\label{eq:segura}
\frac{m}{2} + \sqrt{\frac{m^2}{4}+x^2} \, \leq \, x \frac{K_{\frac{m}{2}+1}(x)}{K_{\frac{m}{2}}(x)} \, \leq \, \frac{m+1}{2} + \sqrt{\frac{(m+1)^2}{4}+x^2}.
\end{align}
For $m\neq0$, inequality \eqref{eq:segura} is strict.

\section{Variation of the Robin Eigenvalue}\label{secA2}
In this section, we partly omit, in favor of readability certain arguments of a function if it is obvious from the context.

\begin{proof}[\proofname\ of Theorem \ref{satz:dotlambdaext}]
The proof is similar to the calculations of the variation of the Robin eigenvalue on bounded domains given in \cite[Section 9.1.2 and Section 9.1.3]{Wagner1}. We use the normalization $|| u(t)||_{L^2(\Omega_t^\text{ext})}=1$, which implies
\begin{align}\label{eq:normierung}
0 = \frac{ \mathrm{d}}{\mathrm{d}t} \int_{\Omega^{\text{ext}}} \widetilde{u}^2(t) J(t) \, \mathrm{d}x = \int_{\Omega^{\text{ext}}} 2 \widetilde{u}(t) \dot{\widetilde{u}}(t) J(t) + \widetilde{u}^2(t) \dot{J}(t) \, \mathrm{d}x.
\end{align}
Therefore, it holds
\begin{align*}
\lambda_1^\alpha(t) &= \int_{\Omega^{\text{ext}}} \left( \nabla \widetilde{u}(t) \right)^T A(t) \nabla \widetilde{u}(t)  \, \mathrm{d}x + \alpha \int_{\partial \Omega} \widetilde{u}^2(t) m(t)\, \mathrm{d}S.
\end{align*}
If we differentiate this expression with respect to $t$, we obtain
\begin{align*}
\dot{\lambda}_1^\alpha(t) =& \int_{\Omega^{\text{ext}}} 2 \left( \nabla \dot{\widetilde{u}}(t) \right)^T A(t) \nabla \widetilde{u}(t) + \left( \nabla \widetilde{u}(t) \right)^T \dot{A}(t) \nabla \widetilde{u}(t)  \, \mathrm{d}x \\
&+\alpha \int_{\partial \Omega} 2 \dot{\widetilde{u}}(t) \widetilde{u}(t) m(t) + \widetilde{u}^2(t) \dot{m}(t)\, \mathrm{d}S.
\end{align*} 
In order to eliminate expressions containing $\dot{\widetilde{u}}$, we multiply \eqref{eq:problemtrafo} with $\dot{\widetilde{u}}$ and integrate over $\Omega^{\text{ext}}$, which yields
\begin{align*}
\lambda_1^\alpha(t) \int_{\Omega^{\text{ext}}} \dot{\widetilde{u}}(t) {\widetilde{u}}(t) J(t) \, \mathrm{d}x = - \int_{\Omega^{\text{ext}}} \dot{\widetilde{u}}(t) L_A {\widetilde{u}}(t) \, \mathrm{d}x.
\end{align*}
In view of $  L_A \widetilde{u} = \operatorname{div}(A \nabla \widetilde{u})$, we can use integration by parts and obtain
\begin{align*}
\lambda_1^\alpha(t) \int_{\Omega^{\text{ext}}} \dot{\widetilde{u}}(t) {\widetilde{u}}(t) J(t) \, \mathrm{d}x &= \int_{\partial \Omega} \dot{\widetilde{u}}(t) \langle A(t) \nabla {\widetilde{u}}(t), \nu \rangle\, \mathrm{d}S + \int_{\Omega^{\text{ext}}} \langle A(t) \nabla {\widetilde{u}}(t), \nabla \dot{\widetilde{u}}(t) \rangle \, \mathrm{d}x \\
&= \int_{\partial \Omega} \dot{\widetilde{u}}(t) \partial_{\nu_{A(t)}} {\widetilde{u}}(t) \, \mathrm{d}S + \int_{\Omega^{\text{ext}}} (\nabla \dot{\widetilde{u}}(t))^T A(t) \nabla {\widetilde{u}}(t) \, \mathrm{d}x \\
&= \alpha \int_{\partial \Omega} \dot{\widetilde{u}}(t)  m(t) \widetilde{u}(t)\, \mathrm{d}S + \int_{\Omega^{\text{ext}}} (\nabla \dot{\widetilde{u}}(t))^T A(t) \nabla {\widetilde{u}}(t) \, \mathrm{d}x,
\end{align*}
where we inserted the boundary condition in the last equality. Thus, we obtain
\begin{align*}
\dot{\lambda}_1^\alpha(t) =& \int_{\Omega^{\text{ext}}} \left( \nabla \widetilde{u}(t) \right)^T \dot{A}(t) \nabla \widetilde{u}(t) + 2 \lambda_1^\alpha(t)  \dot{\widetilde{u}}(t) {\widetilde{u}}(t) J(t)  \, \mathrm{d}x +\alpha \int_{\partial \Omega}  \widetilde{u}^2(t) \dot{m}(t)\, \mathrm{d}S .
\end{align*} 
Using the normalization \eqref{eq:normierung}, we obtain \eqref{eq:lambdadot1}.

For the second variation, we differentiate \eqref{eq:lambdadot1} with respect to $t$ and obtain
\begin{align*}
\ddot{\lambda}_1^\alpha(t) =& \int_{\Omega^{\text{ext}}} \left( \nabla \widetilde{u}(t) \right)^T \ddot{A}(t) \nabla \widetilde{u}(t) + 2 \left( \nabla \dot{\widetilde{u}}(t) \right)^T \dot{A}(t) \nabla \widetilde{u}(t)  \, \mathrm{d}x  \\
&+\alpha \int_{\partial \Omega}  \widetilde{u}^2(t) \ddot{m}(t) +2\widetilde{u}(t)\dot{\widetilde{u}}(t) \dot{m}(t)\, \mathrm{d}S \\
&- \dot{\lambda}_1^\alpha(t) \int_{\Omega^{\text{ext}}} {\widetilde{u}}^2(t)  \dot{J}(t) \, \mathrm{d}x - \lambda_1^\alpha(t) \int_{\Omega^{\text{ext}}} {\widetilde{u}}^2(t)  \ddot{J}(t) + 2 {\widetilde{u}}(t) \dot{\widetilde{u}}(t)  \dot{J}(t) \, \mathrm{d}x.
\end{align*} 
To simplify this expression, we differentiate \eqref{eq:problemtrafo} with respect to $t$, which leads to
\begin{align}
\begin{cases}
L_{\dot{A}(t)} \widetilde{u}(t) + L_{A(t)} \dot{\widetilde{u}}(t) + \dot{\lambda}_1^\alpha(t) \widetilde{u}(t) J(t) + \lambda_1\alpha(t) \left[ \dot{\widetilde{u}}(t) J(t) +  \widetilde{u}(t) \dot{J}(t) \right] = 0  & \, \text{ in } \Omega^{\text{ext}}, \\
-\partial_{\nu_{\dot{A}(t)}} \widetilde{u}(t) -\partial_{\nu_{A(t)}} \dot{\widetilde{u}}(t)  + \alpha \dot{m}(t) \widetilde{u}(t) + \alpha m(t) \dot{\widetilde{u}}(t) = 0 & \,  \text{ on } \partial \Omega.
\end{cases} \nonumber
\end{align}
Multiplying this equation with $\dot{\widetilde{u}}$ and integrating over $\Omega^{\text{ext}}$ yields
\begin{align*}
0=&\int_{\Omega^{\text{ext}}} \dot{\widetilde{u}}(t)  \operatorname{div}(\dot{A}(t) \nabla \widetilde{u}(t)) +\dot{\widetilde{u}}(t) \operatorname{div}(A(t) \nabla \dot{\widetilde{u}}(t) ) \, \mathrm{d}x \\
&+ \int_{\Omega^{\text{ext}}} \dot{\widetilde{u}}(t) \dot{\lambda}_1^\alpha(t) \widetilde{u}(t) J(t) + \dot{\widetilde{u}}(t) \lambda_1^\alpha(t) \left[ \dot{\widetilde{u}}(t) J(t) +  \widetilde{u}(t) \dot{J}(t) \right]  \, \mathrm{d}x.
\end{align*}
After integration by parts and using the boundary condition, we obtain
\begin{align*}
0=&  -\int_{\partial \Omega} \dot{\widetilde{u}}(t)  \left[ \alpha \dot{m}(t) \widetilde{u}(t) + \alpha m(t) \dot{\widetilde{u}}(t) \right]\, \mathrm{d}S \\
&- \int_{\Omega^{\text{ext}}} ( \nabla \dot{\widetilde{u}}(t))^T \dot{A}(t) \nabla {\widetilde{u}}(t) + (\nabla \dot{\widetilde{u}}(t))^T {A}(t) \nabla \dot{\widetilde{u}}(t)  \, \mathrm{d}x  \\
&+ \int_{\Omega^{\text{ext}}} \dot{\widetilde{u}}(t) \dot{\lambda}_1^\alpha(t) \widetilde{u}(t) J(t) + \dot{\widetilde{u}}(t) \lambda_1^\alpha(t) \left[ \dot{\widetilde{u}}(t) J(t) +  \widetilde{u}(t) \dot{J}(t) \right]  \, \mathrm{d}x.
\end{align*}
Finally, adding this two times to $\ddot{\lambda}_1^\alpha$ yields the claimed formula for the second variation.
\end{proof}

\begin{proof}[\proofname\ of Corollary \ref{koro:firstvarlambdat0}]
From Theorem \ref{satz:dotlambdaext}, we obtain
\begin{align*}
\dot{\lambda}_1^\alpha(0) =& \int_{\Omega^{\text{ext}}} \left( \nabla u \right)^T \dot{A}(0) \nabla u - \lambda_1^\alpha(0)  u^2  \dot{J}(0) \,  \mathrm{d}x +\alpha \int_{\partial \Omega}  u^2 \dot{m}(0)\,  \mathrm{d}S.
\end{align*}
It holds $\dot{J}(0) = \operatorname{div}(v)$ and from \cite[(2.3.19)]{Wagner1} we conclude for Hadamard perturbations $\dot{m}(0) = (n-1) H \langle v, -\nu \rangle$. Furthermore, in \cite[Lemma 4.1]{Wagner1} it is shown that $\dot{A}_{i,j}(0) = \operatorname{div}(v) \delta_{i,j} - \partial_j v_i - \partial_i v_j$, where $\partial_i$ is short for $\partial_{x_i}$. Hence, it holds
\begin{align*}
& \int_{\Omega^{\text{ext}}} \left( \nabla u \right)^T \dot{A}(0) \nabla u  \, \mathrm{d}x -2  \int_{\Omega^{\text{ext}}}   \Delta u \langle \nabla u, v \rangle \, \mathrm{d}x \\
=&  \int_{\Omega^{\text{ext}}} \operatorname{div}(v) | \nabla u|^2   \, \mathrm{d}x - 2 \int_{\Omega^{\text{ext}}}  \sum_{i,j=1}^n (\partial_i u)   (\partial_j u) (\partial_i v_j)  \, \mathrm{d}x -2  \int_{\Omega^{\text{ext}}}  \Delta u \langle \nabla u, v \rangle \, \mathrm{d}x \\
=&  \int_{\Omega^{\text{ext}}} \operatorname{div}(v) | \nabla u|^2 + \langle v, \nabla ( | \nabla u|^2) \rangle \, \mathrm{d}x  -2  \int_{\Omega^{\text{ext}}} \langle \nabla u, \nabla (\langle \nabla u, v \rangle) \rangle +\Delta u \langle \nabla u, v \rangle \, \mathrm{d}x\\
=& - \int_{\partial \Omega} \langle v, \nu \rangle | \nabla u|^2 - 2\partial_\nu u \langle \nabla u, v \rangle\,  \mathrm{d}S,
\end{align*}
where we used integration by parts in the last step. Therefore, using $\Delta u = - \lambda_1^\alpha(0) u$ in $\Omega^\text{ext}$, we obtain
\begin{align*}
\dot{\lambda}_1^\alpha(0)=& -\int_{\partial \Omega} \langle v, \nu \rangle | \nabla u|^2 - 2 (\partial_\nu u) \langle \nabla u, v \rangle +\alpha u^2 (n-1) H \langle v, \nu \rangle\,  \mathrm{d}S \\
&+ 2 \int_{\Omega^{\text{ext}}} \Delta u \langle \nabla u, v \rangle \, \mathrm{d}x - \lambda_1^\alpha(0) \int_{\Omega^{\text{ext}}}   u^2  \operatorname{div}(v) \, \mathrm{d}x \\
=& -\int_{\partial \Omega} \langle v, \nu \rangle | \nabla u|^2 - 2 (\partial_\nu u) \langle \nabla u, v \rangle +\alpha u^2 (n-1) H \langle v, \nu \rangle\,  \mathrm{d}S \\
&- 2 \lambda_1^\alpha(0) \int_{\Omega^{\text{ext}}}  u \langle \nabla u, v \rangle \, \mathrm{d}x - \lambda_1^\alpha(0) \int_{\Omega^{\text{ext}}}   u^2  \operatorname{div}(v) \, \mathrm{d}x.
\end{align*}
Using again integration by parts, the last integral can be transformed to
\begin{align*}
&-  \int_{\Omega^{\text{ext}}}   u^2  \operatorname{div}(v) \, \mathrm{d}x = \int_{\Omega^{\text{ext}}}   \langle v, \underbrace{\nabla u^2}_{=2 u \nabla u } \rangle \, \mathrm{d}x +  \int_{\partial \Omega} u^2 \langle v, \nu \rangle \,  \mathrm{d}S.
\end{align*}
Under the assumption $v = \langle v, \nu \rangle \nu$ and in view of the boundary condition, it holds that $(\partial_\nu u) \langle \nabla u, v \rangle  = \alpha^2 u^2\langle v,\nu \rangle$. Hence, \eqref{eq:firstvarlambdat0} follows.
\end{proof}

\begin{proof}[\proofname\ of \eqref{eq:ddotlambda0}]
Inserting $t=0$ into Theorem \ref{satz:dotlambdaext} and using $A(0) = I_{n \times n}$ as well as $m(0) = J(0) =1$ yields
\begin{align}
\ddot{\lambda}_1^\alpha(0) =&\int_{B_R^{\text{ext}}} \left( \nabla u \right)^T \ddot{A}(0) \nabla u    -2  |\nabla \dot{\widetilde{u}}(0))|^2  \, \mathrm{d}x \nonumber +\alpha \int_{\partial B_R}  u^2 \ddot{m}(0)   -2     \dot{\widetilde{u}}(0)^2 \, \mathrm{d}S \nonumber \\
&- \underbrace{\dot{\lambda}_1^\alpha(0)}_{=0} \int_{B_R^{\text{ext}}} u^2  \dot{J}(0) - \dot{\widetilde{u}}(0)  u  \, \mathrm{d}x - \lambda_1^\alpha(0) \int_{B_R^{\text{ext}}} u^2  \ddot{J}(0) - 2\dot{\widetilde{u}}(0)^2  \, \mathrm{d}x. \nonumber
\end{align}
Using \eqref{eq:defshape}, i.e. $\dot{\widetilde{u}}(0) = \langle v, \nabla u \rangle + u'$, we obtain
\begin{align*}
\ddot{\lambda}_1^\alpha(0) =&\int_{B_R^{\text{ext}}} \left( \nabla u \right)^T \ddot{A}(0) \nabla u    -2  |\nabla (\langle v, \nabla u \rangle + u')|^2  \, \mathrm{d}x \nonumber \\
&+\alpha \int_{\partial B_R}  u^2 \ddot{m}(0)   -2     (\langle v, \nabla u \rangle + u')^2 \, \mathrm{d}S \\
&- \lambda_1^\alpha(0) \int_{B_R^{\text{ext}}} u^2  \ddot{J}(0) - 2(\langle v, \nabla u \rangle + u')^2  \, \mathrm{d}x. \nonumber
\end{align*}
With $\mathcal{Q}(u')$ as in \eqref{eq:DefQ} and using $\ddot{\mathcal{S}}(0) = \int_{\partial \Omega} \ddot{m}(0)\, \mathrm{d}S$, the second variation becomes
\begin{align}\label{eq:lambdaddotnullkugel}
\ddot{\lambda}_1^\alpha(0) =&\int_{B_R^{\text{ext}}} \left( \nabla u \right)^T \ddot{A}(0) \nabla u    -2  |\nabla (\langle v, \nabla u \rangle)|^2 -4 \langle \nabla (\langle v, \nabla u \rangle), \nabla u' \rangle  \, \mathrm{d}x \nonumber \\
&+\alpha u^2(R) \ddot{S}(0) -2 \alpha \int_{\partial B_R}  \langle v, \nabla u \rangle^2 +2\langle v, \nabla u \rangle u'   \, \mathrm{d}S \nonumber \\
&- \lambda_1^\alpha(0) \int_{B_R^{\text{ext}}} u^2  \ddot{J}(0) - 2(\langle v, \nabla u \rangle^2 +2\langle v, \nabla u \rangle u' )  \, \mathrm{d}x  -2\mathcal{Q}(u').
\end{align}
We transform the remaining domain integrals into boundary integrals. To this end, we proceed similar to \cite[Section 6.5.1]{Wagner1}. By \cite[Lemma 4.1]{Wagner1}, it holds
\begin{align*}
\ddot{A}_{i,j}(0) =& \left[ (\operatorname{div}(v) )^2 - \operatorname{Tr}\left( D_v (D_v)^T \right) + \operatorname{div}(w) \right] \delta_{i,j} -2  \operatorname{div}(v)  \left[ (\partial_j v_i) +(\partial_i v_j) \right] \\
&+2 \sum_{k=1}^n (\partial_k v_i) (\partial_j v_k) + (\partial_k v_j) (\partial_i v_k) + (\partial_k v_i) (\partial_k v_j) - (\partial_i w_j) - (\partial_j w_i).
\end{align*}
Therefore, it is straight forward to verify that
\begin{align*}
\left( \nabla u \right)^T \ddot{A}(0) \nabla u  =&\left[ (\operatorname{div}(v) )^2 - \operatorname{Tr}\left( D_v (D_v)^T \right) + \operatorname{div}(w) \right] | \nabla u|^2 -4  \operatorname{div}(v)  \langle \nabla u, D_v \nabla u \rangle  \\
&+4  \langle  (D_v)^T \nabla u, D_v \nabla u \rangle+  2 \langle D_v \nabla u, D_v \nabla u \rangle -2 \langle \nabla u, D_w \nabla u \rangle.
\end{align*}
Additionally, it holds
\begin{align*}
|\nabla (\langle v, \nabla u \rangle)|^2 &= \langle D_v \nabla u, D_v \nabla u \rangle + 2 \langle D_v \nabla u, (D^2u) v \rangle + \langle (D^2u) v,(D^2u) v \rangle,
\end{align*}
where $D^2u$ is the matrix given by $(D^2u)_{i,j} = \partial_i \partial_j u$. Furthermore, it holds
\begin{align*}
&\langle \nabla (\langle v, \nabla u \rangle), \nabla u' \rangle = \operatorname{div}(\langle v, \nabla u \rangle \nabla u') - \langle v, \nabla u \rangle \Delta u'.
\end{align*}
Thus, we obtain
\begin{align*}
&\left( \nabla u \right)^T \ddot{A}(0) \nabla u    -2  |\nabla (\langle v, \nabla u \rangle)|^2 -4 \langle \nabla (\langle v, \nabla u \rangle), \nabla u' \rangle  \\
=&\left[ (\operatorname{div}(v) )^2 - \operatorname{Tr}\left( D_v (D_v)^T \right) + \operatorname{div}(w) \right] | \nabla u|^2 -4  \operatorname{div}(v)  \langle \nabla u, D_v \nabla u \rangle  \\
&+4  \langle  (D_v)^T \nabla u, D_v \nabla u \rangle+  2 \langle D_v \nabla u, D_v \nabla u \rangle -2 \langle \nabla u, D_w \nabla u \rangle \\
&- 2 \left[ \langle D_v \nabla u, D_v \nabla u \rangle + 2 \langle D_v \nabla u, (D^2u) v \rangle + \langle (D^2u) v,(D^2u) v \rangle  \right] \\
&- 4 \left[ \operatorname{div}(\langle v, \nabla u \rangle \nabla u') - \langle v, \nabla u \rangle \Delta u' \right] \\
=& \left[ (\operatorname{div}(v) )^2 - \operatorname{Tr}\left( D_v (D_v)^T \right)  \right] | \nabla u|^2 -2 \langle \nabla u, D_w \nabla u \rangle  + \operatorname{div}(w)| \nabla u|^2   \\
&-4  \operatorname{div}(v)  \langle \nabla u, D_v \nabla u \rangle +4  \langle  (D_v)^T \nabla u, D_v \nabla u \rangle \\
&- 4 \langle D_v \nabla u, (D^2u) v \rangle -2 \langle (D^2u) v,(D^2u) v \rangle  \\
&- 4  \operatorname{div}(\langle v, \nabla u \rangle \nabla u') +4 \langle v, \nabla u \rangle \Delta u' .
\end{align*}
We examine the terms separately. To this end, we define
\begin{align*}
E_1:=& \left[ (\operatorname{div}(v) )^2 - \operatorname{Tr}\left( D_v (D_v)^T \right)  \right] | \nabla u|^2,\\
E_2:=& -2 \langle \nabla u, D_w \nabla u \rangle  + \operatorname{div}(w)| \nabla u|^2,\\
E_3:=&-4  \operatorname{div}(v)  \langle \nabla u, D_v \nabla u \rangle +4  \langle  (D_v)^T \nabla u, D_v \nabla u \rangle \\
&- 4 \langle D_v \nabla u, (D^2u) v \rangle -2 \langle (D^2u) v,(D^2u) v \rangle,\\
E_4:=&- 4  \operatorname{div}(\langle v, \nabla u \rangle \nabla u') +4 \langle v, \nabla u \rangle \Delta u'.
\end{align*}
A straightforward calculation shows that
first, it holds that
\begin{align*}
E_1  =& \operatorname{div} \left( [v\operatorname{div}(v) - (D_v)^T v] |\nabla u|^2  \right) -2 \operatorname{div}(v) \langle v, D^2u \nabla u \rangle +2 \langle (D_v)^T v, D^2u \nabla u \rangle, \\
E_2 =&\operatorname{div} \left( w | \nabla u|^2 -2 \nabla u \langle w, \nabla u \rangle \right) +2\Delta u \langle w, \nabla u \rangle,
\end{align*}
and for $E_3$, we use the identity
\begin{align*}
&-\operatorname{div}(v)  \langle \nabla u, D_v \nabla u \rangle  -  \langle D_v \nabla u, (D^2u) v =\operatorname{div} \left[ \nabla u \langle v, D_v \nabla u \rangle - v \langle \nabla u, D_v \nabla u \rangle \right] \\
& +\langle (D^2u) v, (D_v)^T \nabla u \rangle -\langle (D_v)^T \nabla u, D_v \nabla u \rangle -  \Delta u \langle v, D_v \nabla u \rangle - \langle (D_v)^T v, (D^2u) \nabla u \rangle.
\end{align*}
Hence, $E_1 + E_3$ equals
\begin{align*}
&\operatorname{div} \left( [v\operatorname{div}(v) - (D_v)^T v] |\nabla u|^2  \right) +4  \operatorname{div} \left[ \nabla u \langle v, D_v \nabla u \rangle - v \langle \nabla u, D_v \nabla u \rangle \right] \\
&- 4 \Delta u \langle v, D_v \nabla u \rangle +2 \operatorname{div} \left[ \nabla u \langle (D^2u)v,v \rangle \right] -2 \Delta u \langle v, (D^2u)v \rangle  \\
&-2 \operatorname{div} \left[ v \langle (D^2 u) v, \nabla u \rangle \right].
\end{align*}
In total, we obtain that $E_1+E_2+E_3+E_4$ equals
\begin{align*}
&\operatorname{div} \left[ [v\operatorname{div}(v) - (D_v)^T v] |\nabla u|^2 +4   \left( \nabla u \langle v, D_v \nabla u \rangle - v \langle \nabla u, D_v \nabla u \rangle \right) \right]  \\
&+2 \operatorname{div} \left[ - v \langle (D^2 u) v, \nabla u \rangle + \nabla u \langle (D^2u)v,v \rangle -2\left[ \langle v, \nabla u \rangle \nabla u' \right]\right]  \\
&+\operatorname{div} \left[ w | \nabla u|^2 -2 \nabla u \langle w, \nabla u \rangle \right] \\
&  +2\Delta u \langle w, \nabla u \rangle +4 \langle v, \nabla u \rangle \Delta u' -2 \Delta u \langle v, (D^2u)v \rangle- 4 \Delta u \langle v, D_v \nabla u \rangle.
\end{align*}
Using integration by parts, we obtain that $\int_{B_R^{\text{ext}}} E_1+E_2+E_3+E_4 \, \mathrm{d}x$ equals
\begin{align*}
& - \int_{ \partial B_R} |\nabla u|^2 \langle v \operatorname{div}(v) - (D_v)^T v , \nu \rangle +4\langle v, D_v \nabla u \rangle \langle \nabla u , \nu \rangle - 4 \langle \nabla u, D_v \nabla u \rangle \langle v , \nu \rangle\, \mathrm{d}S \\
&-2\int_{ \partial B_R} \langle (D^2 u) v, \nabla u \rangle \langle - v  , \nu \rangle + \langle (D^2u)v,v \rangle  \langle \nabla u , \nu \rangle -2 \langle v, \nabla u \rangle \langle  \nabla u' , \nu \rangle\, \mathrm{d}S \\
&-\int_{ \partial B_R} | \nabla u|^2 \langle w , \nu \rangle  -2 \langle w, \nabla u \rangle \langle\nabla u , \nu \rangle\, \mathrm{d}S\\
&+2\int_{B_R^{\text{ext}}}\Delta u \langle w, \nabla u \rangle +2 \langle v, \nabla u \rangle \Delta u' - \Delta u \langle v, (D^2u)v \rangle- 2 \Delta u \langle v, D_v \nabla u \rangle  \, \mathrm{d} x .
\end{align*}
Since $u$ is radial, it holds $\nabla u = (\partial_\nu u) \nu$. Thus, $\langle v, D_v \nabla u \rangle \partial_\nu u -  \langle \nabla u, D_v \nabla u \rangle \langle v , \nu \rangle = 0$ and we obtain
\begin{align*}
&\int_{B_R^{\text{ext}}} E_1+E_2+E_3+E_4 \, \mathrm{d}x \\
=& - \int_{ \partial B_R} |\nabla u|^2 \langle v \operatorname{div}(v) - (D_v)^T v , \nu \rangle\, \mathrm{d}S \\
&-2\int_{ \partial B_R} \langle (D^2 u) v, \nabla u \rangle \langle - v  , \nu \rangle + \langle (D^2u)v,v \rangle  \langle \nabla u , \nu \rangle -2 \langle v, \nabla u \rangle \langle  \nabla u' , \nu \rangle\, \mathrm{d}S \\
&-\int_{ \partial B_R} | \nabla u|^2 \langle w , \nu \rangle  -2 \langle w, \nabla u \rangle \langle\nabla u , \nu \rangle\, \mathrm{d}S\\
&+2\int_{B_R^{\text{ext}}}\Delta u \langle w, \nabla u \rangle +2 \langle v, \nabla u \rangle \Delta u' - \Delta u \langle v, (D^2u)v \rangle- 2 \Delta u \langle v, D_v \nabla u \rangle \, \mathrm{d}x.
\end{align*}
For the last integral in \eqref{eq:lambdaddotnullkugel}, we use $\ddot{J}(0) = \operatorname{div} \left[ v \operatorname{div}(v) - (D_v)^T v + w \right]$ and 
\begin{align*}
\operatorname{div} \left[ v u \langle \nabla u, v \rangle \right] - \langle \nabla u,v \rangle^2 -u \langle (D^2u)v,v \rangle -u \langle v, D_v \nabla u \rangle= u \langle \nabla u, v \rangle \operatorname{div}(v).
\end{align*}
Therefore, we obtain
\begin{align*}
u^2 \ddot{J}(0) =& u^2 \operatorname{div} \left[ v \operatorname{div}(v) - (D_v)^T v + w \right] \\
=&\operatorname{div} \left[ u^2 (v \operatorname{div}(v) - (D_v)^T v  + w )  \right] -2u \left[ \langle \nabla u, v \operatorname{div}(v) -(D_v)^T v +w \rangle   \right] \\
=&\operatorname{div} \left[ u^2 (v \operatorname{div}(v) - (D_v)^T v  + w )  \right] -2u \left[ \langle \nabla u,  -(D_v)^T v +w \rangle   \right] \\
&-2 \left[ \operatorname{div} \left[ v u \langle \nabla u, v \rangle \right] - \langle \nabla u,v \rangle^2 -u \langle (D^2u)v,v \rangle -u \langle v, D_v \nabla u \rangle \right].
\end{align*}
Thus, it follows
\begin{align*}
&\int_{B_R^{\text{ext}}} u^2  \ddot{J}(0) - 2(\langle v, \nabla u \rangle^2 +2\langle v, \nabla u \rangle u' )  \, \mathrm{d}x \\
=& -\int_{ \partial B_R}  u^2 \langle   v \operatorname{div}(v) - (D_v)^T v  + w  , \nu \rangle\, \mathrm{d}S -2 \int_{B_R^{\text{ext}}}  u \langle \nabla u, v \rangle \operatorname{div}(v) \, \mathrm{d}x \\
&-2\int_{B_R^{\text{ext}}}u  \langle \nabla u,  -(D_v)^T v +w \rangle +  \langle v, \nabla u \rangle^2 +2\langle v, \nabla u \rangle u' \, \mathrm{d}x   \\
=& -\int_{ \partial B_R}  u^2 \langle   v \operatorname{div}(v) - (D_v)^T v  + w  , \nu \rangle\, \mathrm{d}S  +2 \int_{ \partial B_R}  u \langle \nabla u, v \rangle \langle v, \nu \rangle\, \mathrm{d}S \\
&+2 \int_{B_R^{\text{ext}}} \langle v, \nabla \left( u \langle \nabla u, v \rangle \right) \, \mathrm{d}x  \\
&-2\int_{B_R^{\text{ext}}}u  \langle \nabla u,  -(D_v)^T v +w \rangle +  \langle v, \nabla u \rangle^2 +2\langle v, \nabla u \rangle u' \, \mathrm{d}x.
\end{align*}
Together with
\begin{align*}
\langle v, \nabla \left( u \langle \nabla u, v \rangle \right) \rangle &=  \langle \nabla u, v \rangle^2 + u \langle (D^2 u) v, v \rangle + u \langle v, D_v \nabla u \rangle,
\end{align*}
we obtain
\begin{align*}
&\int_{B_R^{\text{ext}}} u^2  \ddot{J}(0) - 2(\langle v, \nabla u \rangle^2 +2\langle v, \nabla u \rangle u' )  \, \mathrm{d}x\\
=&-\int_{ \partial B_R}  u^2 \langle   v \operatorname{div}(v) - (D_v)^T v  + w  , \nu \rangle  -2  u \langle \nabla u, v \rangle \langle v, \nu \rangle\, \mathrm{d}S \\
&+2 \int_{B_R^{\text{ext}}}   u \langle (D^2 u) v, v \rangle + u \langle v, D_v \nabla u \rangle  - \left[ u  \langle \nabla u,  -(D_v)^T v +w \rangle   +2\langle v, \nabla u \rangle u' \right] \, \mathrm{d}x \\
=&-\int_{ \partial B_R}  u^2 \langle   v \operatorname{div}(v) - (D_v)^T v  + w  , \nu \rangle  -2  u \langle \nabla u, v \rangle \langle v, \nu \rangle\, \mathrm{d}S \\
&+2 \int_{B_R^{\text{ext}}}   u \langle (D^2 u) v, v \rangle + 2 u  \langle v, D_v \nabla u \rangle \, \mathrm{d}x  -2\int_{B_R^{\text{ext}}}u  \langle \nabla u,  w \rangle   +2\langle v, \nabla u \rangle u' \, \mathrm{d}x.
\end{align*}
Inserting this into \eqref{eq:lambdaddotnullkugel} and using $\Delta u = - \lambda_1^\alpha u$ yields
\begin{align*}
\ddot{\lambda}_1^\alpha(0)  =&- \int_{ \partial B_R} |\nabla u|^2 \left[ \operatorname{div}(v) \langle v , \nu \rangle -\langle   v , (D_v) \nu \rangle +  \langle w , \nu \rangle \right]\, \mathrm{d}S \\
&-2\int_{ \partial B_R} \langle (D^2 u) v, \nabla u \rangle \langle - v  , \nu \rangle + \langle (D^2u)v,v \rangle  \partial_\nu u -2 \langle v, \nabla u \rangle \langle  \nabla u' , \nu \rangle\, \mathrm{d}S \\
&+2\int_{ \partial B_R}    \langle w, \nabla u \rangle \partial_\nu u\, \mathrm{d}S  +\alpha u^2(R) \ddot{S}(0) -2 \alpha \int_{\partial B_R}  \langle v, \nabla u \rangle^2 +2\langle v, \nabla u \rangle u'   \, \mathrm{d}S  \\
&+ \lambda_1^\alpha(0)  \int_{ \partial B_R}  u^2 \langle   v \operatorname{div}(v) - (D_v)^T v  + w  , \nu \rangle  -2  u \langle \nabla u, v \rangle \langle v, \nu \rangle\, \mathrm{d}S    \\
&+2 \lambda_1^\alpha(0)   \int_{B_R^{\text{ext}}} 2\langle v, \nabla u \rangle u' \, \mathrm{d}x   +2\int_{B_R^{\text{ext}}} 2 \langle v, \nabla u \rangle \Delta u'  \, \mathrm{d}x -2\mathcal{Q}(u')\\
=&- \int_{ \partial B_R} \left[ |\nabla u|^2  - \lambda_1^\alpha(0) u^2 \right] \left[ \operatorname{div}(v) \langle v , \nu \rangle -\langle   v , (D_v) \nu \rangle +  \langle w , \nu \rangle \right]\, \mathrm{d}S \\
&-2\int_{ \partial B_R} - \langle (D^2 u) v, \nabla u \rangle \langle  v  , \nu \rangle + \langle (D^2u)v,v \rangle  \partial_\nu u -2 \langle v, \nabla u \rangle \langle  \nabla u' , \nu \rangle\, \mathrm{d}S \\
&+2\int_{ \partial B_R}    \langle w, \nabla u \rangle \partial_\nu u\, \mathrm{d}S  +\alpha u^2(R) \ddot{S}(0) -2 \alpha \int_{\partial B_R}  \langle v, \nabla u \rangle^2 +2\langle v, \nabla u \rangle u'   \, \mathrm{d}S  \\
&-\lambda_1^\alpha(0)  \int_{ \partial B_R} 2  u \langle \nabla u, v \rangle \langle v, \nu \rangle\, \mathrm{d}S    \\
&+2 \lambda_1^\alpha(0)   \int_{B_R^{\text{ext}}} 2\langle v, \nabla u \rangle u' \, \mathrm{d}x   +2\int_{B_R^{\text{ext}}} 2 \langle v, \nabla u \rangle \Delta u'  \, \mathrm{d}x -2\mathcal{Q}(u') .
\end{align*}
Using
\begin{align*}
 &\langle (D^2u)v,v \rangle  \partial_\nu u  - \langle (D^2 u) v, \nabla u \rangle \langle  v  , \nu \rangle \\
 =&  \langle (D^2u)v, \langle v, \nu \rangle \nu \rangle  \partial_\nu u  - \langle (D^2 u) v, (\partial_\nu u) \nu \rangle \langle  v  , \nu \rangle = 0,
\end{align*}
the second variation becomes
\begin{align*}
\ddot{\lambda}_1^\alpha(0)  =&- \int_{ \partial B_R} \left[ |\nabla u|^2  - \lambda_1^\alpha(0) u^2 \right] \left[ \operatorname{div}(v) \langle v , \nu \rangle -\langle   v , (D_v) \nu \rangle +  \langle w , \nu \rangle \right]\, \mathrm{d}S \\
&- 2 \lambda_1^\alpha(0)  \int_{ \partial B_R}   u \langle \nabla u, v \rangle \langle v, \nu \rangle\, \mathrm{d}S +4 \int_{\partial B_R}  \langle v, \nabla u \rangle \left[ \partial_\nu u' - \alpha u' \right]\, \mathrm{d}S  \\
&-2 \alpha \int_{\partial B_R}  \langle v, \nabla u \rangle^2    \, \mathrm{d}S  +2\int_{ \partial B_R}    \langle w, \nabla u \rangle \partial_\nu u\, \mathrm{d}S +\alpha u^2(R) \ddot{S}(0)  \\
&+4  \int_{B_R^{\text{ext}}} \langle v, \nabla u \rangle \left[ \lambda_1^\alpha(0) u' +  \Delta u' \right] \, \mathrm{d}x   -2\mathcal{Q}(u').
\end{align*}
In view of Lemma \ref{lemma:u'Kugel}, $u'$ satisfies
\begin{align*}
\begin{cases}
\Delta u' + \lambda_1^\alpha u' = 0   &\text{ in } B_R^{\text{ext}} ,\\
-\partial_\nu u' + \alpha u' = -K u(R) \langle v, \nu \rangle &\text{ on } \partial B_R.
\end{cases}
\end{align*}
Hence, we obtain
\begin{align*}
\ddot{\lambda}_1^\alpha(0)  =&- \int_{ \partial B_R} \left[ |\nabla u|^2  - \lambda_1^\alpha(0) u^2 \right] \left[ \operatorname{div}(v) \langle v , \nu \rangle -\langle   v , (D_v) \nu \rangle +  \langle w , \nu \rangle \right]\, \mathrm{d}S \\
&- 2 \lambda_1^\alpha(0)  \int_{ \partial B_R}   u \langle \nabla u, v \rangle \langle v, \nu \rangle\, \mathrm{d}S +4 K \int_{\partial B_R}  \langle v, \nabla u \rangle   \langle v, \nu \rangle   u \, \mathrm{d}S  \\
&-2 \alpha \int_{\partial B_R}  \langle v, \nabla u \rangle^2    \, \mathrm{d}S  +2\int_{ \partial B_R}    \langle w, \nabla u \rangle \partial_\nu u\, \mathrm{d}S +\alpha u^2(R) \ddot{S}(0)  -2\mathcal{Q}(u').
\end{align*}
Using that $u$ and $ \partial_\nu u$ are constant on $\partial B_R$ as well as $\partial_\nu u = \alpha u$, we have
\begin{align*}
\ddot{\lambda}_1^\alpha(0)  =&- \left[ \alpha^2 u^2(R)  - \lambda_1^\alpha(0) u^2 \right] { \int_{ \partial B_R}   \operatorname{div}(v) \langle v , \nu \rangle -\langle   v , (D_v) \nu \rangle +  \langle w , \nu \rangle \, \mathrm{d}S} \\
&- 2 \lambda_1^\alpha(0) \alpha u^2(R) \int_{ \partial B_R}    \langle v, \nu \rangle^2  \, \mathrm{d}S +4 K \alpha u^2(R)  \int_{\partial B_R}   \langle v, \nu \rangle^2 \, \mathrm{d}S  \\
&-2 \alpha^3 u^2(R)  \int_{\partial B_R}  \langle v, \nu \rangle^2     \, \mathrm{d}S  +2 \alpha^2 u^2(R) \int_{ \partial B_R}    \langle w,  \nu \rangle \, \mathrm{d}S +\alpha u^2(R) \ddot{S}(0)  -2\mathcal{Q}(u').
\end{align*}
The first integral vanishes because of
\begin{align*}
\int_{ \partial B_R}   \operatorname{div}(v) \langle v , \nu \rangle -\langle   v , (D_v) \nu \rangle +  \langle w , \nu \rangle \, \mathrm{d}S =&- \int_{  B_R^{ext}} \ddot{J}(0) \, \mathrm{d}x  = \ddot{V}(0) =  0.
\end{align*}
Thus, we obtain
 \begin{align*}
\ddot{\lambda}_1^\alpha(0)  =&2 u^2(R) \alpha \left[-  \lambda_1^\alpha(0) +2K-\alpha^2 \right] \int_{ \partial B_R}    \langle v, \nu \rangle^2  \, \mathrm{d}S \\
&+ 2 \alpha^2 u^2(R) \int_{ \partial B_R}    \langle w,  \nu \rangle \, \mathrm{d}S +\alpha u^2(R) \ddot{S}(0)  -2\mathcal{Q}(u').
\end{align*}
In \cite[(5.56)]{Henrot}, A. Henrot and M. Pierre show that for $f \in \mathcal{C}^1(\partial \Omega)$, it holds 
\begin{align*}
\int_{\partial B_R} f \operatorname{div}_{\partial \Omega} (v)\, \mathrm{d}S = \int_{\partial B_R} - \langle v, \nabla^\tau f \rangle +(n-1) f \langle v, \nu \rangle H \, \mathrm{d}S,
\end{align*}
where $\operatorname{div}_{\partial \Omega} (v) := \left[ \operatorname{div}(v) - \langle (D_v) \nu, \nu \rangle \right]$ and $\nabla^\tau f := \nabla f - \langle \nabla f, \nu \rangle \nu $. Since we consider a Hadamard perturbation, it holds
\begin{align*}
\langle v, \nabla^\tau f \rangle = \langle v, \nabla f \rangle - \langle v, \langle \nabla f, \nu \rangle \nu \rangle = \langle \langle v, \nu \rangle \nu, \nabla f \rangle - \langle v, \langle \nabla f, \nu \rangle \nu \rangle = 0.
\end{align*}
Since $\nu$ is the outer normal of $B_R$, we have $H= \frac{1}{R}$. Thus, for $f = \langle v, \nu \rangle$, we obtain,
\begin{align}\label{eq:partinttang}
\int_{\partial B_R} \langle v, \nu \rangle \operatorname{div}_{\partial \Omega} (v)\, \mathrm{d}S = \int_{\partial B_R} (n-1) \langle v, \nu \rangle^2 \frac{1}{R}\, \mathrm{d}S.
\end{align}
Because of $\ddot{V}(0) = 0$, it holds that $0 = \int_{\partial B_R}  \operatorname{div}(v) \langle v , \nu \rangle -\langle   v , (D_v) \nu \rangle +  \langle w , \nu \rangle \, \mathrm{d}S$. Hence, we obtain
\begin{align*}
\int_{ \partial B_R}    \langle w,  \nu \rangle \, \mathrm{d}S &= - \int_{ \partial B_R}     \operatorname{div}(v) \langle v , \nu \rangle -\langle   v , (D_v) \nu \rangle \, \mathrm{d}S \\
&= - \int_{ \partial B_R}    \left[ \operatorname{div}_{\partial B_R}(v) + \langle  D_v \nu, \nu \rangle \right] \langle v , \nu \rangle -\langle   \langle v, \nu \rangle \nu , (D_v) \nu \rangle \, \mathrm{d}S \\
&=-\int_{\partial B_R}  \frac{(n-1) \langle v, \nu \rangle^2}{R}\, \mathrm{d}S - \int_{ \partial B_R}    \langle  D_v \nu, \nu \rangle  \langle v , \nu \rangle - \langle v, \nu \rangle \langle    \nu , (D_v) \nu \rangle \, \mathrm{d}S\\
&=-\int_{\partial B_R}  \frac{(n-1) \langle v, \nu \rangle^2}{R}\, \mathrm{d}S.
\end{align*}
Thus, it finally holds
\begin{align*}
\ddot{\lambda}_1^\alpha(0)  =&2 u^2(R) \alpha \left[-  \lambda_1^\alpha(0) +2K-\alpha^2 \right] \int_{ \partial B_R}    \langle v, \nu \rangle^2  \, \mathrm{d}S  \\
&- 2 \alpha^2 u^2(R) \frac{n-1}{R} \int_{\partial B_R} \langle v, \nu \rangle^2 \, \mathrm{d}S +\alpha u^2(R) \ddot{S}(0)  -2\mathcal{Q}(u')   \\
=&2 u^2(R) \alpha K \int_{ \partial B_R}    \langle v, \nu \rangle^2   \mathrm{d}S +\alpha u^2(R) \ddot{S}(0)  -2\mathcal{Q}(u').
\end{align*}
\end{proof}




\end{appendices}


\bibliography{sn-bibliography}

\end{document}